\documentclass{amsart}
\usepackage[top=3cm,left=3.5cm,right=3.5cm,bottom=2.1cm]{geometry}
\usepackage{amsfonts}
\usepackage{amsmath,amsthm,amscd,mathrsfs,amssymb,graphicx,enumerate,
  dsfont}
\usepackage{xypic}

\newtheorem{thm2}{Theorem}

\newtheorem{defn2}{Definition}
\newtheorem{thm}{Theorem}[section]
\newtheorem{cor}[thm]{Corollary}
\newtheorem{lem}[thm]{Lemma}
\newtheorem{prop}[thm]{Proposition}

\theoremstyle{definition}
\newtheorem{obs}[thm]{Observations}
\newtheorem{defn}[thm]{Definition}

\newtheorem{rmk}[thm]{Remark}

 \DeclareMathOperator{\Spec}{Spec}

\DeclareMathOperator{\Hom}{Hom} 
 \DeclareMathOperator{\rk}{rk}
 
\DeclareMathOperator{\Pic}{Pic}

\newcommand{\C}{\ensuremath\mathds{C}}
\newcommand{\R}{\ensuremath\mathds{R}}
\newcommand{\Z}{\ensuremath\mathds{Z}}
\newcommand{\Q}{\ensuremath\mathds{Q}}

\newcommand{\h}{\ensuremath\mathfrak{h}}

\newcommand{\PP}{\ensuremath\mathds{P}}

\newcommand{\HH}{\ensuremath\mathrm{H}}
\newcommand{\CH}{\ensuremath\mathrm{CH}}

\begin{document}

\thispagestyle{empty} 
\title{Exceptional collections, and the N\'eron--Severi lattice for surfaces}
\author{Charles Vial}


	\thanks{The  author was supported by the Fund for Mathematics at the
		Institute for Advanced Study and by EPSRC Early Career Fellowship
		EP/K005545/1.}


\address{DPMMS, University of Cambridge, Wilberforce Road, Cambridge
  CB3 0WB, UK} \email{c.vial@dpmms.cam.ac.uk}


\begin{abstract}
We work out properties of smooth projective varieties $X$ over a (not
necessarily algebraically closed) field $k$ that admit collections of objects in
the bounded derived category of coherent sheaves $D^b(X)$ that are either full
exceptional, or numerically exceptional of maximal length. Our main result gives
a necessary and sufficient condition on the N\'eron--Severi lattice for a smooth
projective surface $S$ with $\chi(O_S)=1$ to admit a numerically exceptional
collection of maximal length, consisting of line-bundles. As a consequence we determine exactly which
complex surfaces with $p_g=q=0$ admit a numerically exceptional collection of
maximal length. Another consequence is that a minimal geometrically rational surface
with a numerically exceptional collection of maximal length is rational.
\end{abstract}

\maketitle

\section*{Introduction}
Recently, a substantial amount of work \cite{ao, bgs, bgks, gs}
was carried
out in order to exhibit exceptional collections of line-bundles of maximal
length on complex surfaces of general type with $p_g=q=0$, motivated by the will
to exhibit geometric (quasi)-phantom triangulated categories, \emph{i.e.}, categories
with trivial or torsion Grothendieck group $K_0$\,; see also \cite{go}. Kuznetsov's  recent
ICM address \cite{kuznetsov2} sets the notion of exceptional collection into the
wider picture of semi-orthogonal decompositions for bounded derived categories
of smooth projective varieties.

The purpose of this work is twofold\,: we give arithmetic, and geometric,
constraints for the existence of exceptional collections of maximal length on smooth projective surfaces over a field. For
instance, on the geometric side, we show that there are no numerical
obstructions to the existence of exceptional collections of maximal length on
complex surfaces of general type with $p_g=q=0$ (in fact we give in Theorem
\ref{thm surface
general type} a complete classification of complex surfaces with $p_g=q=0$ that
admit a numerically exceptional collection of maximal length). On the arithmetic
side, we show that a minimal
geometrically rational surface over a perfect field $k$ that  admits a
numerically exceptional collection of maximal length is rational (Theorem
\ref{thm
rational}).  We also show that a numerically exceptional collection of maximal
length, consisting of line-bundles, on a surface $S$ defined
over an arbitrary field $k$ remains of maximal length after any field extension
(Theorem \ref{thm cycle}). These results are deduced from a general criterion
(the main Theorem
\ref{thm criterion}) that gives, for a smooth projective surface $S$ defined
over an arbitrary field $k$, a necessary and sufficient condition on the
N\'eron--Severi
lattice of $S$ for $S$ to admit a numerically exceptional collection of maximal
length, consisting of line-bundles. 

Along the way, we find that if a surface admits a full exceptional collection,
then its integral Chow motive is a direct sum of Lefschetz motives (Theorem \ref{thm
surface}). 
On a slightly unrelated note, we also provide a new characterization of
projective space (Theorem \ref{thm projective}).

\subsection*{Derived category of coherent sheaves and exceptional objects} Let
$k$ be a field and let $\mathcal{T}$ be a $k$-linear triangulated category. The
typical example of triangulated category that we have in mind is given by the
bounded derived category $D^b(X)$ of coherent sheaves on a smooth projective
variety $X$ defined over  $k$. Given a morphism $f : T_1 \rightarrow T_2$
between two objects $T_1$ and $T_2$ of $\mathcal{T}$, there is a distinguished
triangle $$T_1 \stackrel{f}{\longrightarrow} T_2 \longrightarrow
\mathrm{cone}(f) \longrightarrow T_1[1].$$
If $\mathcal{A}$ and $\mathcal{B}$ are two strictly full triangulated
subcategories of $\mathcal{T}$ (we mean that $\mathcal{A}$ and $\mathcal{B}$ are
closed under shifts and cones), then $$\mathcal{T} = \langle
\mathcal{A},\mathcal{B}\rangle$$ is a \emph{semi-orthogonal decomposition} if
\begin{enumerate}[(i)]
\item $\Hom(B,A) = 0$ for all objects $A \in \mathcal{A}$ and $B\in \mathcal{B}$
(note that since $\mathcal{A}$ and $\mathcal{B}$ are closed under shifts, we in
fact have $\mathrm{Ext}^i(B,A) = 0$ for all integers $i\in \Z$)\,;
\item $\mathcal{A}$ and $\mathcal{B}$ generate $\mathcal{T}$\,: for all objects
$T\in \mathcal{T}$, $T$ fits into a  distinguished triangle $T_\mathcal{B}
\rightarrow T \rightarrow T_\mathcal{A} \rightarrow T_\mathcal{B}[1]$ for some
objects $T_\mathcal{A}$ of $\mathcal{A}$ and $T_\mathcal{B}$ of $\mathcal{B}$.
\end{enumerate}
More generally, $$\mathcal{T} = \langle
\mathcal{A}_1,\ldots,\mathcal{A}_n\rangle$$ is a \emph{semi-orthogonal
decomposition} if 
\begin{enumerate}[(i)]
\item $\Hom(\mathcal{A}_i,\mathcal{A}_j) = 0$ for all $i>j$\,;
\item For all $T\in \mathcal{T}$, there exist $T_i \in \mathcal{T}$ and a
sequence $0=T_n \rightarrow T_{n-1} \rightarrow \ldots \rightarrow T_1
\rightarrow T_0=T$ such that $\mathrm{cone}(T_i \rightarrow T_{i-1}) \in
\mathcal{A}_i$ for all $i$.
\end{enumerate}
\medskip

The simplest triangulated category is probably the bounded derived category
$D^b(k\mathrm{-vs})$ of $k$-vector spaces. Given a triangulated category
$\mathcal{T}$, it is natural to be willing to split off copies of
$D^b(k\mathrm{-vs})$ inside $\mathcal{T}$, in the sense of semi-orthogonal
decompositions. Consider then a functor $D^b(k\mathrm{-vs}) \rightarrow
\mathcal{T}$. Such a functor is determined by the image $E \in \mathcal{T}$ of
the one-dimensional $k$-vector space placed in degree $0$. Let us denote this
functor $\varphi_E$\,; then, for any complex $V^\bullet \in D^b(k\mathrm{-vs})$, we
have $\varphi_E(V^\bullet) = V^\bullet \otimes E$. A right-adjoint functor is
given by $\varphi_E^!(F) = \Hom^\bullet(E,F)$, so that
$\varphi_E^!\varphi_E(V^\bullet) = \Hom^\bullet(E,E) \otimes V^\bullet$. Thus
$\varphi_E$ is fully faithful if and only if $\Hom^\bullet(E,E) = k$ placed in
degree $0$.

\begin{defn2}
An object $E \in \mathcal{T}$ is \emph{exceptional} if $$\Hom(E,E[l]) = \left\{
\begin{array}{cc}
             k & \mbox{if} \ l=0\,; \\
             0 & \mbox{otherwise}.  
 \end{array} \right. $$
 
 \noindent An \emph{exceptional collection} is a collection of exceptional
objects $E_1,\ldots, E_n \in \mathcal{T}$ such that $$\Hom(E_i,E_j[l]) = 0 \quad
\mbox{for all} \ i>j \ \mbox{and all} \ l\in \Z.$$
\end{defn2}
An important feature of exceptional objects is that strictly full triangulated
subcategories generated by exceptional collections are \emph{admissible},
meaning that the inclusion functor admits a left and a right adjoint.
Consequently, given an exceptional collection $(E_1,\ldots , E_n)$ of objects in
$\mathcal{T}$, we have a semi-orthogonal decomposition $$\mathcal{T} = \langle
E_1,\ldots, E_n,\mathcal{A}\rangle,$$ where $\mathcal{A} = \{T\in \mathcal{T} :
\Hom(T,E_i) = 0 \ \mbox{for all} \ 1 \leq i\leq n \}$, and
where, by abuse, we have denoted $E_i$ the subcategory generated by $E_i$.

\noindent When $\mathcal{A} = 0$, the semi-orthogonal decomposition
$$\mathcal{T} = \langle E_1,\ldots, E_n\rangle$$ is said to be a \emph{full
exceptional collection}.
\medskip

Let us now give some examples of smooth projective varieties $X$ whose bounded
derived category of coherent sheaves admits a full exceptional collection. 
\begin{itemize}
\item $X=\PP^n$\,: we have the Beilinson sequence \cite{beilinson} (this is
perhaps the most famous example of a full exceptional collection) $$D^b(\PP^n) =
\langle O,O(1),\ldots,O(n)\rangle.$$ 
\item $X= \widetilde{\PP}^2$ the blow-up of $\PP^2$ at a point\,: if $E$ denotes
the exceptional divisor, then $$D^b(\widetilde{\PP}^2) = \langle
O_E(-1),O(-2),O(-1),O\rangle$$ is a full exceptional collection. More
generally, there is a blow-up formula due to Orlov \cite{orlov}.
\item $X=Q^n \subset \PP^{n+1}_\C$\,: a smooth complex quadric. Kapranov
\cite{kapranov} showed that $$D^b(Q^n) = \left\{ \begin{array}{cc}
            \langle S,O,O(1),\ldots,O(n-1)\rangle & \mbox{if} \ n \ \mbox{is
odd} \\
              \langle S^-,S^+,O,O(1),\ldots,O(n-1)\rangle & \mbox{if} \ n \
\mbox{is even}
 \end{array} \right. $$ is a full exceptional collection. Here, $S,S^-$ and
$S^+$ are certain spinor bundles.
\end{itemize}
Other examples of varieties admitting exceptional collections include complex
Grassmannians (Kapranov \emph{loc. cit.}), and several other complex rational homogeneous spaces. In fact, it is expected that if $G$ is a semi-simple algebraic group over an algebraically closed field of characteristic zero and $P \subset G$ is a parabolic subgroup, then
there is a full exceptional collection of vector bundles in $D^b(G/P)$\,; see \cite[Conjecture 1.1]{kp} and the results in that direction therein.
The projective space $\PP^n$ admits a
full exceptional collection consisting of $n+1$ \emph{line-bundles}. At least
in characteristic zero,
this
property characterizes completely the projective space among $n$-dimensional
smooth projective varieties\,:
\begin{thm2}[Theorem \ref{thm projective}] \label{thm2 projective}Let $X$ be a
	smooth projective
	variety of dimension $n$ over a field $k$ of characteristic zero. Assume that
	$ \langle
	L_0,\ldots,L_{n}\rangle$ is a full exceptional collection of
	$D^b(X)$ for some line-bundles $L_0,\ldots, L_{n}$.
	Then $X$ is isomorphic to the projective space $\PP^n$.
\end{thm2}

As is apparent, the class of varieties admitting full exceptional collections is
rather restricted.
Perhaps the simplest constraint for a smooth projective variety to admit a full
exceptional collection is the following. If we have a semi-orthogonal
decomposition $\mathcal{T} = \langle \mathcal{A}, \mathcal{B}\rangle$, then
$K_0(\mathcal{T}) = K_0(\mathcal{A})\oplus K_0(\mathcal{B})$. Since
$K_0(D^b(k\mathrm{-vs})) = \Z$, we see that if $D^b(X)=\langle E_1,\ldots,
E_r\rangle$ is a full exceptional collection, then $K_0(X)=\Z^r$. As will be
explained in the introduction of Section 2, this implies that the Chow motive of
$X$ with rational coefficients is a direct sum of Lefschetz motives. Such a constraint was originally obtained \emph{via} the theory of non-commutative motives by Marcolli and Tabuada \cite{mt}.

\subsection*{Chow motives and Lefschetz motives} Let $R$ be a ring. The category of
Chow motives with $R$-coefficients over $k$ is constructed as follows. First one
linearizes the category of smooth projective varieties over $k$ by declaring
that $\Hom(X,Y) = \CH^{d}(X\times Y)\otimes_Z R$, that is, by declaring that the
morphism between $X$ and $Y$ are given by correspondences with $R$-coefficients
modulo rational equivalence. Here, $X$ is assumed to be of pure dimension $d$
(otherwise one works component-wise) and the composition law is given by $$\beta
\circ \alpha = (p_{XZ})_*(p_{XY}^*\alpha \cdot p_{YZ}^*\beta) \in \CH^d(X\times
Z),$$ for all $\alpha \in \CH^d(X\times Y)$ and all $\beta \in \CH^e(Y\times
Z)$. Here, $e$ is the dimension of $Y$, and $p_{XZ}$, $p_{XY}$, $p_{YZ}$ are the
projections from $X\times Y\times Z$ onto $X\times Z, X\times Y, Y\times Z$,
respectively. This $R$-linear category  is far from being abelian, so that one
formally adds to this $R$-linear category the images of idempotents. This is
called taking the pseudo-abelian, or Karoubi, envelope. This new category is
called the category of effective Chow motives, and objects are pairs $(X,p)$,
where $X$ is a smooth projective variety of dimension $d$ and $p\in
\CH^d(X\times X)\otimes_\Z R$ is an idempotent. When $p$ is the class of the
diagonal $\Delta_X$ in $\CH^d(X\times X)$, we write $\h(X)$ for $(X,\Delta_X)$. 
In general, the object $(X,p)$ should be thought of as the image of $p$ acting
on the motive $\h(X)$ of $X$. For example, in the 
category of effective Chow motives, the motive $\h(\PP^1)$  of the
projective line 
becomes isomorphic to $(\PP^1,p) \oplus (\PP^1,q)$, where $p:= \{0\} \times
\PP^1$ and $q:=\PP^1 \times \{0\}$ are idempotents in $\Hom(\h(\PP^1),\h(\PP^1))
:= \CH^1(\PP^1\times \PP^1)$. The object $(\PP^1,p)$ is isomorphic to $\mathds{1} := \h(\Spec
k)$, and the motive $(\PP^1,p)$ is
called the \emph{Lefschetz motive} and is written $\mathds{1}(-1)$. The fiber product of two smooth projective varieties induces a tensor
product in the category of effective Chow motives, for which $\mathds{1}$ is a unit. The category of Chow motives
with $R$-coefficients is then obtained by formally inverting the Lefschetz
motive.

Concretely, a Chow motive with $R$-coefficients is a triple $(X,p,n)$ consisting
of a smooth projective variety $X$ of pure dimension $d$ over $k$, of a
correspondence $p \in \CH^d(X\times X) \otimes_Z R$ such that $p\circ p = p$,
and of an integer $n$. A morphism $\gamma \in \Hom((X,p,n),(Y,q,m))$ is an
element of $q\circ (\CH^{d-n+m}(X\times Y)\otimes_\Z R)\circ p$.\medskip

The simplest motives are the motive $\mathds{1}:=\h(\Spec k)$ of a point $\Spec
k$ and its Tate twists,
that is, the motives $\mathds{1}(n) = (\Spec k, \mathrm{id}, n)$ for $n\in \Z$.
These are called
the \emph{Lefschetz motives}. Note that $$\Hom(\mathds{1}(-n),\h(X)) =
\CH^{n}(X)\otimes_\Z R.$$
As in the case of triangulated categories, it is natural, given a motive $M$, to
split off copies of Lefschetz motives. Given a morphism $\gamma \in
\Hom(\mathds{1}(-n),\h(X))$, there is an obvious obstruction to the existence
of a splitting to that morphism\,: if $\gamma \in \CH^{n}(X)$ is a
non-zero numerically trivial cycle, then $\gamma$ does not admit a left-inverse.
Even if $\gamma$ is not numerically trivial, the existence of a left-inverse is
in general a problem of existence of algebraic cycles (in this case, the existence of a cycle $\gamma' \in \CH^{d-n}(X)$ such that $\deg (\gamma \cdot \gamma')=1$). In Section 2, we prove\,:
\begin{thm2} [Theorem \ref{thm surface}]\label{thm2 1} Let $S$ be a smooth
projective surface
over a
field $k$. Assume that $S$ has a full exceptional collection.
  Then the integral Chow motive of $S$ is isomorphic to a direct sum of Lefschetz
  motives.
\end{thm2} 
One may naturally ask if the converse to Theorem \ref{thm surface} holds.
In Remark \ref{rmk barlow}, we give evidence that the converse should fail to
be true.

\subsection*{Numerical constraints} Although the problem of classifying smooth
projective complex surfaces that admit a full exceptional collection seems out
of reach at present (it is conjectured that only the surfaces that are rational have a
full
exceptional collection), a fair amount of work \cite{ao, bgs, bgks, gs} has been
carried out in order to
construct exceptional collections of maximal length on complex surfaces with
$p_g=q=0$. (As usual, for a smooth projective surface $S$, the geometric genus is $p_g := h^0(\Omega_S^2) = h^2(O_S)$ and the irregularity is $q :=  h^1(O_S)$.) A first step in constructing exceptional collections consists in
 constructing numerically exceptional collections.

Recall that, given a $k$-linear triangulated category $\mathcal{T}$ and two objects $E$ and $F$ in $\mathcal{T}$, the \textit{Euler pairing} $\chi$ is the integer $$\chi(E,F) := \sum_l
(-1)^l \dim_k \Hom(E, F[l]).$$ The Euler pairing defines a bilinear pairing on the Grothendieck groups $K_0(\mathcal{T})$ that we still denote $\chi$.
\begin{defn2}
An object $E$  is said to be \emph{numerically exceptional} if
$$\chi(E,E)= 1.$$
A collection $(E_1,\ldots,E_r)$ of numerically exceptional objects in $\mathcal{T}$
is
called a \emph{numerically exceptional collection} if $$\chi(E_j,E_i)  = 0 \quad \mbox{for all} \ j>i.$$ 
A numerically exceptional
collection $(E_1,\ldots,E_r)$ on a smooth projective variety $X$ over $k$ is said to be of \emph{maximal length} if $E_1,\ldots,E_r$ span the numerical Grothendieck group, or equivalently if $r$ is
equal to the rank of $K^{\mathrm{num}}_0(X)$. (Here, $K_0^{\mathrm{num}}(X) := K_0(X)/(\ker \chi)$\,; note that the left and right kernels of $\chi$ are the same so that the notation
$\ker \chi$ is unambiguous.)
\end{defn2}
Obviously, an exceptional object is numerically
exceptional, an exceptional collection is a numerically exceptional collection,
and a full exceptional collection is a numerically exceptional collection of
maximal length.\medskip

In this work, we give a complete classification of smooth projective complex surfaces, with
$p_g=q=0$, that admit numerically exceptional collections of maximal length\,:
\begin{thm2}[Theorem \ref{thm surface general type}]\label{thm2 2} Let $S$ be a
smooth projective complex surface, with
$p_g=q=0$. As usual, $\kappa(S)$ denotes the Kodaira dimension of $S$.
\begin{itemize}
\item If $S$ is not minimal, then $S$ has a numerically exceptional collection
of maximal length.
\end{itemize}
Assume now that $S$ is minimal.
\begin{itemize}
\item If $\kappa(S) = -\infty$, then $S$ has a numerically
exceptional collection of maximal length.
\item If $\kappa(S) = 0$, then $S$ is an Enriques surface and
it does not have a
numerically exceptional collection of maximal length.
\item If $\kappa(S)=1$, then $S$ is a Dolgachev surface
$X_9(p_1,\ldots,p_n)$, and $S$ has a numerically exceptional collection of
maximal length if and only if $S$ is one of $X_9(2,3)$, $X_9(2,4)$, $X_9(3,3)$,
$X_9(2,2,2)$. (We refer to paragraph \ref{sec cpxsurface} for the notations.)
\item If  $\kappa(S) = 2$, then $S$ has a numerically
exceptional collection of maximal length.
\end{itemize}
\end{thm2}
In particular, for surfaces of general type with $p_g=q=0$, there is no
numerical obstruction to the existence of exceptional collections of maximal
length.  In the case of Enriques surfaces, a general Enriques surface admits ten
different elliptic pencils $|2F_1|,\ldots |2F_{10}|$ and the line-bundles
$O(F_1),\ldots, O(F_{10})$ provide an exceptional collection of length $10$\,; see
 Zube \cite{zube}. (Any reordering of this length-$10$ exceptional collection  is still an exceptional
collection.)  Our Theorem \ref{thm
surface general type} says in particular that it is not possible to find an
exceptional collection consisting of 12 exceptional objects. On the other hand, Theorem \ref{thm
	surface general type} says that an Enriques surface blown up at  a point admits a numerically exceptional collection of maximal length, that is, there is no numerical obstruction to the existence  of an exceptional collection of maximal length on an Enriques surface blown up at a point\,; \emph{cf.} Remark \ref{R:blowup}. \medskip

So far, exceptional collections have almost only be considered for varieties
defined over algebraically closed fields. Our analysis of numerically
exceptional collections of maximal length leads, intuitively, to the conclusion
that exceptional collections of maximal length, consisting of rank one objects, do not exist for surfaces that
are not ``split''. First, we have a general result\,:
\begin{thm2}[Theorem \ref{thm cycle}] \label{thm2 3}
Let $S$ be a smooth projective surface over a
field $k$, with $\chi(O_S)=1$ and with $ \HH_{et}^1(S_{\bar{k}},\Q_\ell) = 0$.
Here, $\bar{k}$ is a separable
closure of $k$, $S_{\bar{k}} = S\times_{\mathrm{Spec}\,
k}{\mathrm{Spec} \,\bar{k}}$, and $\ell$ is  a prime $\neq \mathrm{char} \, k$.
Assume that
$S$ admits a numerically
  exceptional collection $(E_0,E_1,\ldots, E_{n+1})$ of maximal length,  consisting of rank one objects. Then the cycle
class
map $$\CH^1(S) \otimes \Z_\ell \rightarrow
\HH_{et}^2(S_{\bar{k}},\Z_\ell(1))$$ is
surjective modulo torsion, that is, it induces a surjective map $$\CH^1(S)
\otimes \Z_\ell \twoheadrightarrow
\HH_{et}^2(S_{\bar{k}},\Z_\ell(1))/torsion.$$
In particular,
for all field extensions $K/k$, the
 collection $((E_0)_K,(E_1)_K,\ldots, (E_{n+1})_K)$ for $S_K = S
\times_{\mathrm{Spec}\, k} \mathrm{Spec}\, K$ is numerically
    exceptional of maximal
length, and the base-change $$\mathrm{N}^1(S) \stackrel{\simeq}{\longrightarrow}
\mathrm{N}^1(S_K)$$ is an isometry.
\end{thm2}
Second, we find arithmetic obstructions for geometrically rational surfaces to admit a numerically exceptional collection of maximal length\,:
\begin{thm2} [Theorem \ref{thm rational}] \label{thm2 4}
Let $S$ be a minimal smooth projective surface defined over a perfect field
$k$, such that $S_{\bar
k}$ is rational. If $S$ admits a
numerically exceptional collection of maximal length, then $S$ is rational.
\end{thm2}

In particular, the folklore conjecture of Orlov stating that a surface over an
algebraically closed field admits a full exceptional collection only if
it is rational can be extended to surfaces defined over arbitrary
fields. Note however that a surface may be rational but not admit a numerically exceptional collection of maximal length\,; see Remark \ref{R:rational}.
\medskip

Surprisingly,  Theorems \ref{thm2 1},  \ref{thm2 2},  \ref{thm2 3} and \ref{thm2
4} are obtained essentially by exploiting
the linear algebra 
constraints on the N\'eron--Severi lattice of $S$ imposed by the existence of a
numerically exceptional collection of maximal length. The rich linear algebra stemming from the Riemann--Roch formula is treated independently in the appendix.
In this work, the \emph{N\'eron--Severi lattice} of a smooth projective surface
$S$ over a field $k$, denoted $\mathrm{N}^1(S)$, refers to the
group of codimension-$1$ cycles of $S$ modulo numerical equivalence. Note that,
by definition of numerical equivalence, $\mathrm{N}^1(S)$ is
torsion-free.\medskip

Our main theorem, from which all theorems stated above arise from, is\,:
\begin{thm2}[Main Theorem \ref{thm criterion}]
 Let $S$ be a smooth projective surface over a
field $k$, with $\chi(O_S)=1$.  
The
following statements are equivalent\,:
\begin{enumerate}[(i)]
\item $S$ admits a numerically
  exceptional collection $(L_0,L_1,\ldots, L_{n+1})$ of line-bundles which is
  of \emph{maximal length}, that is, $n = \rk \mathrm{N}^1(S)$.
\item  $S$ admits a numerically
  exceptional collection $(E_0,E_1,\ldots, E_{n+1})$ of rank one objects in $D^b(S)$
which is
  of \emph{maximal length}, that is, $n = \rk \mathrm{N}^1(S)$.
\item We have $(K_S)^2=10-  \rk \mathrm{N}^1(S)$, and the lattice
$\mathrm{N}^1(S)$ 
  and the canonical divisor $K_S$, when seen as an element of
$\mathrm{N}^1(S)$, satisfy one of the following properties\,:
  \begin{itemize}
\item $\mathrm{N}^1(S)$ is unimodular of rank $1$, and $K_S = 3D$ for some
primitive divisor $D$\,;
\item   $\mathrm{N}^1(S)$ is the hyperbolic plane, and $K_S = 2D$ for some
primitive divisor $D$\,; 
\item  $\mathrm{N}^1(S)$ is unimodular and odd of rank $>1$, and $K_S$ is
a primitive divisor.
\end{itemize}
  \end{enumerate}
\end{thm2}

Note that a complex smooth projective surface with $p_g=q=0$ always satisfies the equation $K_S^2=10- \rk \operatorname{N}^1(S)$. It is thus surprising that, over any field, the existence of a numerically exceptional collection of line-bundles of maximal length imposes that equation.
\medskip

A recent result of Perling \cite{perling} (see Theorem \ref{thm perling}) implies that, for complex surfaces with $p_g=q=0$, the conditions of Theorem \ref{thm criterion} are further equivalent to the existence of a numerically exceptional collection of maximal length (without any assumptions on the ranks of the objects of the collection).
Theorem \ref{thm criterion}  suggests then that if an exceptional collection of maximal length exists on a surface over an algebraically closed field, then an exceptional collection of maximal length that consists of line-bundles should exist. (As far as I am aware, all examples of complex surfaces that admit an exceptional collections of maximal length admit such a collection consisting of line-bundles.)\medskip

Finally, given
a smooth projective variety $X$, note that if $X$ has an exceptional object of non-zero rank or a numerically
exceptional line-bundle, then the
structure sheaf $O_X$ is numerically exceptional, that is, $\chi(O_X)=1$. (Note also that, since the classes of rank $0$ objects sit in a codimension $1$
subspace of $K_0(X)$, if  $X$ has an exceptional collection of maximal length, then $X$ has an exceptional object of non-zero rank and hence $\chi(O_X)=1$.)
Therefore, the running assumption that $\chi(O_S)=1$ is innocuous.

\subsection*{Notations} Given a smooth projective surface $S$ defined over a field $k$, $\operatorname{N}^1(S)$ denotes the N\'eron--Severi lattice of $S$. The following numerical invariants of $S$ will be used\,:

$\rho = \rk \operatorname{N}^1(S)$, the Picard rank of $S$.

$q = \dim_k H^1(S,O_S)$, the irregularity of $S$.

$p_g = \dim_k H^2(S,O_S)$, the geometric genus of $S$.

$b_i = \dim_{\Q_\ell} H_{et}^i(S_{\bar k},\Q_\ell)$, the Betti numbers of $S$, where $\ell$ is a prime $\neq \mathrm{char}\; k$. 

\subsection*{Acknowledgments} Thanks to Magdalene College, Cambridge, the Simons Center for Geometry and Physics, and the
Institute for Advanced Study for excellent working conditions. 
Thanks to Marcello Bernardara, Fabrizio Catanese, Dmitri Orlov, Markus Perling, and Burt Totaro
for useful discussions.  Thanks to Pierre Deligne for his
interest and for fruitful discussions \cite{deligne} related to the  formula \eqref{eq kappa}. Many thanks to Alexander Kuznetsov and to the referees for their insightful
comments.

\section{A characterization of projective space}

In this section, upon which the rest of this paper does not depend,  the base
field $k$ is assumed to be of characteristic
zero.  Galkin, Katzarkov, Mellit, and Shinder \cite{gkms} have recently
considered so-called \emph{minifolds}. These are smooth projective varieties of
dimension $n$ whose bounded derived category $D^b(X)$ admits a full exceptional
collection of objects in $D^b(X)$ of length $n+1$.

\begin{thm}\label{thm gkms} [Galkin, Katzarkov, Mellit, and Shinder] Assume that
the base field $k$ is the field of complex numbers. Then
\begin{enumerate}
\item The only two-dimensional minifold is $\PP^2$.
\item The minifolds of dimension $3$ are\,: the projective space $\PP^3$, the
quadric three-fold, the del Pezzo quintic three-fold, and a six-dimensional
family of three-folds $V_{22}$\,; see \cite{gkms} for more details.
\item The only four-dimensional Fano minifold is $\PP^4$.
\end{enumerate}
\end{thm}

For other examples of minifolds, we refer to \cite{samokhin, kuzminifold}.
 Here, although we allow the base-field $k$ to be non-algebraically closed, we
consider a more
restrictive class of varieties, namely smooth projective varieties of dimension
$n$ whose bounded derived category $D^b(X)$ admits a full exceptional collection
of line-bundles of length $n+1$. We show in Theorem \ref{thm projective} below
that such a property characterizes
completely projective space. As Theorem \ref{thm gkms} shows, it is important to
consider full exceptional collections of line-bundles rather than full
exceptional collections consisting of objects in the derived category $D^b(X)$.
In fact, it is important to consider full
exceptional collections of
line-bundles, rather than merely full exceptional collections of pure
sheaves or even vector-bundles. For example, Kapranov \cite{kapranov}
showed that quadrics of odd dimension, say $d$, over an algebraically closed
field have a full exceptional
collection consisting of $d+1$ vector-bundles. On a slightly different
perspective, Bernardara
\cite{bernardara} showed
that a Severi--Brauer variety $X$ of dimension $r$ has a full semi-orthogonal
collection of objects $E_i$, $0 \leq i \leq r$, such that $\Hom(E_i,E_i[l])= 0$
for all $l\neq 0$ and $\Hom(E_i,E_i) =A^{\otimes i}$
for the central division algebra $A$ over $k$ that has same class as $X$ in the
Brauer group $\mathrm{Br}(k)$.

\begin{thm} \label{thm projective} Let $X$ be a smooth projective
  variety of dimension $n$ over a field $k$ of characteristic zero. Assume that
$ \langle
  L_0,\ldots,L_{n}\rangle$ is a full exceptional collection of
  $D^b(X)$ for some line-bundles $L_0,\ldots, L_{n}$.
  Then $X$ is isomorphic to the projective space $\PP^n$.
\end{thm}
\begin{proof}
  First, note that the assumption implies that $\Pic X$ is torsion-free (see
Lemma
\ref{lem torsionfree} below) and of rank $1$, so that $\Pic X = \Z H$ for
  some divisor $H$.
By Bondal--Polishchuk \cite[Theorem 3.4]{bp}, a $d$-dimensional smooth
projective variety with a full exceptional
collection consisting
of $d+1$ pure sheaves is necessarily Fano, that is, $-K_X$ is ample. The
theorem then follows from Proposition \ref{prop projective} below. 
\end{proof}

\begin{prop} \label{prop projective} Let $X$ be a smooth projective
  variety of dimension $n$ with $\mathrm{Pic} \, X = \Z H$. Assume
  either that the dimension $n$ of $X$ is odd, or that $X$ is Fano.
 If $X$ admits a numerically exceptional collection $(
  L_0,\ldots,L_n)$ of line-bundles, then $X \cong \PP^n$.
\end{prop}
\begin{proof}
  Since $\Pic X = \Z H$, we may write $L_i = O_X(a_iH)$ for some
  integers $a_i$. By Riemann--Roch, $\chi(O_X(aH))$ is a polynomial
  $P$ with rational coefficients of degree $n$ in the variable
  $a$. Since $\chi(O_X) =1$, this polynomial vanishes at most $n$
  times. We know, by semi-orthogonality, that $\chi(L_j,L_i) = \chi(L_i\otimes (L_j)^{-1}) =
  P(a_i-a_j) = 0$ for all $i<j$. Therefore the set $\{a_i - a_j :
  0\leq i < j \leq n\}$ has order at most $n$. Moreover, $a_i \neq
  a_j$ for all $i\neq j$, because otherwise $\chi(O_X(a_iH - a_jH)) =
  \chi(O_X)=1$ would not be zero. This easily implies that there exist
  an integer $m$ and a non-zero integer $k$ such that $a_i = m+ki$ for
  all $0 \leq i \leq n$.  Up to replacing $H$ with $-H$ if necessary,
  we may assume that $k$ is a positive integer. Now the polynomial $P$
  vanishes exactly at $-lk$ for $1 \leq l \leq n$. By Riemann--Roch,
  we have
\begin{equation}\label{eq RR}
  P(a) = \chi(O_X(aH)) = \frac{\deg(H^n)}{n!}a^n + \frac{\deg(H^{n-1}
    \cdot c_1(X))}{2(n-1)!} a^{n-1} + \cdots + \chi(O_X).
\end{equation} Therefore, $n! = \deg(H^n)\cdot k \cdot (2k) \cdots
(nk) =\deg(H^n)k^n n!$.  Having in mind that $k$ is positive, it
follows that $k=1$ and then that
$\deg (H^n) = 1$.

We can also compute $\deg(H^{n-1}\cdot c_1(X))$\,: by looking at the
coefficient of $a^{n-1}$ in \eqref{eq RR}, we find $\sum_{l=1}^n l =
\frac{n}{2} \deg(H^{n-1}\cdot c_1(X))$, which gives $\deg(H^{n-1}\cdot
c_1(X)) = n+1$. Therefore $c_1(X) = (n+1)H$.  We now distinguish
whether $X$ is odd-dimensional, or Fano. In the latter case, by
definition, $c_1(X)$ is ample. In the odd-dimensional case, since either $H$ or
$-H$
is ample and since $\deg(H^n)=1$, we find that $H$ is ample. In both
cases, $c_1(X)$ is $n+1$ times an ample divisor.
By \cite{ko}, it follows that the base-change of $X$ to the complex
numbers is isomorphic to $\PP_\C^n$. Thus $X$ is a Severi--Brauer
variety. It has a zero-cycle of degree 1, namely $H^n$. Therefore it
is split, \emph{i.e.}, $X$ is isomorphic to $\PP^n$.
\end{proof}

\begin{rmk}
When $n$ is odd, Proposition \ref{prop projective} shows that, provided
$\mathrm{Pic}(X)$ is torsion-free of rank $1$, the
assumptions of Theorem \ref{thm projective} can be relaxed,
essentially by dropping the condition that $\langle
L_0,\ldots,L_n\rangle$ is full. This is possibly related to the fact that there
are no fake projective spaces of odd dimension\,; \emph{cf.} \cite{py}.
As a corollary to Proposition \ref{prop projective}, one sees that a quadric
hypersurface of odd dimension
$n\geq 3$ or a non-split Severi--Brauer variety of odd dimension does
not admit a numerically exceptional collection of line-bundles of length $n+1$.
\end{rmk}

\section{Full exceptional collections on surfaces and Lefschetz motives}

Let $X$ be a smooth projective variety defined over a field $k$. Assume that $X$
has a full exceptional collection. Then, by flat base-change \cite{kuznetsov},
for a universal domain $\Omega$ containing $k$ (this means that $\Omega$ is
algebraically closed of infinite transcendence degree over its prime subfield), 
the pull-back $X_\Omega$ of $X$ along $\Spec \, \Omega \rightarrow \Spec \, k$
also has a full exceptional collection. It follows that $K_0(X_\Omega)$ is of finite
rank (see e.g. Lemma \ref{lem torsionfree}), and hence by applying the Chern
character that the Chow ring $\CH^*(X_\Omega) \otimes_\Z \Q$ is a
finite-dimensional vector space over $\Q$. By Kimura \cite{kimura2}, it follows
that the Chow
motive with rational coefficients of $X_\Omega$ is isomorphic to a direct sum of
Lefschetz motives. (In fact, one may choose such an isomorphism to be defined over
the base field $k$, so that  the Chow motive with rational coefficients of $X$
is isomorphic to a sum of Lefschetz motives\,; see \cite[Corollary 3.5]{vial}). Such a
result was also obtained in \cite{mt}, via the theory of non-commutative
motives. 
The main result of this section is Theorem \ref{thm surface}\,: we
show
that when $S$ is a surface with a full exceptional collection, then the result
above can be improved by showing that the \emph{integral} Chow motive of $S$ is
isomorphic to a sum of Lefschetz motives. 
 The question of understanding the links between the derived category of
coherent sheaves on $X$ and the Chow motive of $X$ is of course not new and we
refer to Orlov's \cite{orlov2}. \medskip

\subsection{The Riemann--Roch formula for surfaces}
In this section, $S$ is a smooth projective surface defined over a
field $k$.  Let $K_S$ be the canonical divisor on
$S$ and let $D$ be a divisor on $S$. The Riemann--Roch formula is
\begin{equation*}
  \chi(O_S(D)) = \frac{1}{2} D\cdot(D-K_S) + \chi(O_S).
\end{equation*}
More generally, there is a Riemann--Roch formula for any object $E$ in $D^b(S)$.
The rank of an object $E$ in $D^b(S)$ is defined as follows\,: if $E^\bullet$ is a complex
representing $E$, then $\rk E := \sum_i (-1)^i \rk E^i$. 
If $E$ and $F$ are two objects in $D^b(S)$ of respective ranks
$e$ and $f$, then the Riemann--Roch formula is
\begin{footnotesize}
$$\chi(E,F) = ef\chi(O_S) + \frac{1}{2}\left(fc_1(E)^2
+ec_1(F)^2-2c_1(E)c_1(F)\right)- \frac{1}{2}
K_S\cdot\left(ec_1(F)-fc_1(E)\right) -  \left(fc_2(E)+ec_2(F)\right).$$
\end{footnotesize}

Assume from now on that $\chi(O_S)=1$.
If $E$ is an
object in $D^b(S)$ of rank $1$ that is numerically exceptional, that is
$\chi(E,E)=1$, then the above Riemann--Roch formula gives $c_2(E) = 0$. 
This justifies referring to  a numerically exceptional object of rank one as a \emph{numerical line-bundle}.
Finally, note that if
$E$ and $F$ are two numerical line-bundles in $D^b(S)$,
then the Riemann--Roch formula takes the simple form\,:
\begin{equation}\label{RR}
  \chi(E,F) = \frac{1}{2}\left(c_1(F)-c_1(E)\right)^2 - \frac{1}{2}
K_S\cdot\left(c_1(F)-c_1(E)\right) + 1.
\end{equation}
\medskip

\subsection{Numerically exceptional collections of line-bundles and the Riemann--Roch formula}
If $(E_0,E_1,\ldots,
E_{r+1})$ is a numerically exceptional collection, then by definition the matrix $\big(\chi(E_i,E_j)\big)_{0\leq i,j \leq r+1}$ is upper-triangular with $1$'s as diagonal entries. If moreover one assumes that the objects $E_i$ have rank $1$ for all $i$, then the Riemann--Roch formula relates the upper-triangularity of the matrix $\big(\chi(E_i,E_j)\big)_{0\leq i,j \leq r+1}$ to the trigonality of the matrix $(D_i\cdot D_j)_{1\leq i,j\leq r+1}$, where $D_i := c_1(E_i) - c_1(E_{i-1})$\,:

\begin{prop} \label{prop intersection matrix}
 Let $S$ be a smooth projective surface with $\chi(O_S)=1$ and let $r$ be a non-negative integer. 
The following statements are equivalent.
\begin{enumerate}[(i)]
\item There exists a collection $(E_0,E_1,\ldots,
E_{r+1})$ of line-bundles which is a numerically exceptional
collection in $D^b(S)$.
\item There exists a collection $(E_0,E_1,\ldots,
E_{r+1})$ of numerical line-bundles in $D^b(S)$ which is a numerically exceptional
collection.
\item There exist divisors $D_1,\ldots, D_{r+1} \in \CH^1(S)$ such that
$K_S\cdot D_i = -2-(D_i)^2$ for all $i$ and such that the
intersection matrix $(D_i\cdot D_j)_{1 \leq i,j \leq r+1}$ has the trigonal form
	$$\big(D_i\cdot D_j\big)_{1\leq i,j \leq r+1} =  \left( \begin{array}{ccccc}
	a_1 & 1 & &    \\
	1 & a_2 & \ddots &    \\
	& \ddots &  \ddots & 1\\
	& &   1 & a_{r+1}\end{array} \right)$$ where all blank entries consist of zeroes.
\end{enumerate}
\end{prop}
\begin{proof} $(i) \Rightarrow (ii)$\,: This is obvious. \medskip
	
	$(ii) \Rightarrow (iii)$\,: Assume that the collection of numerical line-bundles
$(E_0,E_1,\ldots,
E_{r+1})$ is  numerically exceptional and let us  define, for $0<i\leq r+1$, $D_i :=
c_1(E_i)- c_1(E_{i-1}) \in \CH^1(S)$.
  By orthogonality, we have, for all $0 \leq i<j \leq r+1$,
$\chi(E_j,E_i) = 0$.  By Riemann--Roch \eqref{RR}, we find for all $0 \leq i<j \leq r+1$, 
\begin{equation*}\label{E:RR}
(D_{i+1} + D_{i+2} + \cdots+ D_j)^2
+ K_S\cdot (D_{i+1} + D_{i+2} + \cdots +  D_j) = -2.
\end{equation*}
Taking $j=i+1$ yields $K_S\cdot D_i = -2-(D_i)^2$ for all $1 \leq i \leq r+1$. Taking then $j=i+2$ yields $D_i \cdot D_j = 1$ for all $1\leq i,j \leq r+1$ such that $|i-j|=1$. Finally, taking $j=i+3$, $j=i+4$, and so on, gives  $D_i \cdot D_j = 0$ for all $1\leq i,j\leq r+1$ such that $|i-j|>1$. 
\medskip
  
  $(iii) \Rightarrow (i)$\,:
   We define $E_0 := O_S$, and $E_i := O_S(D_1 + \cdots + D_i)$ for all
$1\leq i \leq r+1$. Since $\chi(O_S)=1$, we have
$\chi(E_i,E_i)=1$ for all $i$. On the other hand, 
  by Riemann--Roch \eqref{RR}, one immediately finds
that $\chi(E_j,E_i)=0$ for $0\leq i<j \leq r+1$, thus showing that the collection
$(E_0,E_1,\ldots,
  E_{r+1})$ is  numerically exceptional.
\end{proof}

\subsection{Numerically exceptional collections of maximal length and the N\'eron--Severi lattice}
In order to prove Theorem \ref{thm surface}, we will have to prove that the existence of a 
(numerically) exceptional collection of maximal length consisting of objects in $D^b(S)$ implies that the N\'eron--Severi lattice $\operatorname{N}^1(S)$ is unimodular. For that matter,  M.~Perling \cite[Corollary 10.9 \& Remark 10.12]{perling} recently proved\,:

\begin{thm}[Perling \cite{perling}] \label{thm perling} 
Let $S$ be a smooth projective surface with $\chi(O_S)=1$. Then any numerically
exceptional collection of maximal length on $S$ can be transformed by mutations
into a numerically
exceptional collection of maximal length $(Z_0,Z_1,\ldots,Z_t, F_0,\ldots,F_{\rho-t})$, where $\rk Z_i = 0$ for $0\leq i \leq t$ and $\rk F_j = 1$ for $0\leq j \leq \rho-t$, and $\rho-t = 2$ or $3$.

Moreover, if $p_g=q=0$ and $\rho = b_2$ (or more generally if $\rho = b_2-2b_1$), then any numerically
exceptional collection of maximal length on $S$ can be transformed by mutations
into a numerically exceptional collection of maximal length that consists of rank one objects.
\end{thm}

From the arguments developed by Perling, one can extract the following\,: 

\begin{prop}\label{P:unimodularity}
	Let $S$ be a smooth projective surface with $\chi(O_S)=1$. Assume that $S$ has a numerically
	exceptional collection of maximal length. Then the N\'eron--Severi lattice $\operatorname{N}^1(S)$ is unimodular, and 
	$\rho = b_2 - 2b_1 $ modulo $8$.
\end{prop}
\begin{proof} 
	By Theorem \ref{thm perling}, it suffices to prove that  if $(Z_0,Z_1,\ldots,Z_t, F_0,\ldots,F_{s})$ is a numerically exceptional collection of maximal length, with $\rk Z_i=0$ for $0\leq i \leq t$ and $\rk F_j = 1$ for $0\leq j \leq s$, then the N\'eron--Severi lattice is unimodular.
	
	Suppose $Z$ is a numerically exceptional object of rank zero. From the Riemann--Roch formula, one immediately sees that $c_1(Z)^2=-1$. 
	Suppose now that $(Z_1,Z_2)$ is a numerically exceptional collection consisting of objects of rank zero. From the Riemann--Roch formula again, one finds that $c_1(Z_1)\cdot c_1(Z_2) = 0$. 
	
	Consider now a numerically exceptional collection $(Z_0,Z_1,\ldots,Z_t, F_0,\ldots,F_{s})$, where $\rk Z_i=0$ for $0\leq i \leq t$. The Riemann--Roch formula gives for all $i$ and all $j$ 
	$$2c_1(Z_i) \cdot c_1(F_j) = -\rk (F_j) \big(K_X\cdot c_1(Z_i) +1 + 2c_2(Z_i)\big).$$
Thus we see that
	\begin{enumerate}[(i)]
		\item $c_1(Z_i)^2 = -1$ for all $i$\,;
		\item $c_1(Z_i)\cdot c_1(Z_j) =  0$ for all $i\neq j$\,;
		\item $c_1(Z_i)\cdot (c_1(F_j) - c_1(F_k)) = (\rk F_k - \rk F_j) \big(K_X\cdot c_1(Z_i) +1 + 2c_2(Z_i)\big)$ for all $i$, $j$, and $k$.
	\end{enumerate}
	Therefore, if $(Z_0,Z_1,\ldots,Z_t, F_0,\ldots,F_{s})$ is of maximal length, with $\rk Z_i=0$ for $0\leq i \leq t$ and  $\rk F_j=1$ for $0\leq j \leq s$, then the subspace of $\operatorname{N}^1(S)$ spanned by $c_1(F_0),\ldots, c_1(F_s)$ has rank $s-1$.   Thus, denoting $D_j := c_1(F_j) - c_1(F_{j-1})$, we see that the matrix $\big(D_i\cdot D_j \big)_{1\leq i,j\leq s}$, which is trigonal by the proof of Proposition \ref{prop intersection matrix}, is degenerate. By Proposition \ref{P:unimodular}, $\big(D_i\cdot D_j \big)_{1\leq i,j\leq s-1}$ is unimodular. We conclude that the lattice $$\Z c_1(Z_0) \oplus \cdots \oplus \Z c_1(Z_t) \oplus \Z D_1 \oplus \cdots \oplus \Z D_{s-1}$$ is unimodular, and therefore that $\operatorname{N}^1(S)$ is unimodular.
	
	Finally, we show that $\rho = b_2 - 2b_1$ modulo $8$. Noether's formula $K_S^2  = 12\chi(O_S) - c_2(S)$ gives $K_S^2 = 12 - 2b_0 +2b_1 - b_2 = 10 + 2b_1 - b_2$. On the other hand, since the lattice $\operatorname{N}^1(S)$ is unimodular, van der Blij's lemma (see 	\cite[Lemma
	II.(5.2)]{hm}, or Remark \ref{rmk K2}), together with the Hodge index theorem, gives $K_S^2 = 2-\rho$ modulo 8. This concludes the proof.  
\end{proof}

\begin{rmk} At least in the case when the base field $k$ is algebraically
closed,
Sasha Kuznetsov has mentioned to me the following 
geometric argument, which is directly inspired from \cite[Theorem 4.7]{br}, showing that the existence of a numerically exceptional collection of maximal length implies the unimodularity of the N\'eron--Severi lattice.
Choose a basis of $\mathrm{N}^1(S)$ consisting of classes of smooth curves $C_i$
which
intersect pairwise transversally. Let $F_i$ be a theta-characteristic on $C_i$,
considered as a torsion sheaf on $S$ supported on $C_i$. Then it is clear that
$\dim \mathrm{Ext}^p(F_i,F_j) = [C_i]\cdot [C_j]$ for $p = 1$ and $0$ otherwise,
so
$\chi(F_i,F_j) = -[C_i]\cdot [C_j]$.
Moreover, clearly $\chi(O_S,F_i) = 0$ for all $i$, and if $P$ is a general point
(not lying on any of $C_i$) then $\chi(O_S,O_P) = \chi(O_P,O_S) = 1$,
$\chi(O_P,O_P) = 0$,
and $\chi(F_i,O_P) = \chi(O_P,F_i) = 0$. This shows that the bilinear form
$\chi$ expressed
in the basis $(O_S,O_P,F_1,\dots,F_n)$ is block upper-triangular with the first
of the diagonal blocks being a $2$-by-$2$ matrix $$\left(
\begin{array}{cc}
             * & 1 \\
              1 & 0  \end{array} \right)$$ and the second
diagonal block being the matrix $(-[C_i]\cdot [C_j])_{1\leq i,j \leq n}$. Thus
its
determinant is $(-1)^{n+1}\det([C_i]\cdot [C_j])$.
On the other hand, if there is a numerically exceptional collection then the
determinant is $1$,
so one can conclude that the determinant of the intersection form is
$(-1)^{n+1}$.
\end{rmk}

A straightforward corollary to Proposition \ref{P:unimodularity} is\,:

\begin{cor} \label{cor zerocycle} Let $S$ be a smooth projective surface with
$\chi(O_S)=1$. Assume that the surface $S$ admits a
numerically exceptional collection of maximal length. Then $S$ has a zero-cycle
of degree 1.\qed
\end{cor}

For example, a smooth quadric surface $\subset \PP^3$ with no rational point
does not admit a numerically exceptional collection of maximal length, and hence
does not admit a full exceptional collection. (See Theorem \ref{thm geomrat surfaces} and Remark \ref{R:rational} for stronger statements.)\medskip

Note that the intersection pairing on $\mathrm{N}^1(S)$ for a smooth
projective complex surface $S$ with $p_g =0$ is always
unimodular by Poincar\'e duality and by the Lefschetz $(1,1)$-theorem. Thus
Proposition \ref{P:unimodularity}  provides an arithmetic obstruction
for
a surface defined over a non-algebraically closed field to admit a numerically exceptional collection  of
maximal length.

\subsection{Exceptional collections of maximal length and Chow motives}

\begin{lem}\label{lem torsionfree}
Let $X$ be a smooth projective variety. Assume that $K_0(X)$ is torsion-free.
Then $\CH^1(X)$ is torsion-free. If additionally $X$ is a smooth projective
surface, then the Chow ring $\CH^*(X)$ is torsion-free.
\end{lem}
\begin{proof} First, we note that  $\CH^0(X) = \Z[X]$ is torsion-free for all
varieties $X$.  
 
We show that if a smooth variety $X$ has torsion-free $K_0(X)$, then $\CH^1(X)$
is torsion-free. This was already proved in \cite[Lemma 2.2]{gkms}, and we
reproduce their proof for the sake of completeness. Recall that the first Chern class provides a group isomorphism $\mathrm{Pic}(X) \cong \CH^1(X)$, so that it is equivalent to show that $\mathrm{Pic}(X)$ is torsion-free. Assume that $L$ is a
line-bundle with $L^{\otimes r} = 0$ for some integer $r\geq 2$. The element
$[L]-1 \in K_0(X)$ has rank zero and thus belongs to $F^1K_0(X)$, where
$\mathrm{F}^\bullet$ denotes the topological filtration on $K$-groups. By
multiplicativity of the topological filtration, we find that $([L]-1)^{\dim X
+1} = 0 \in K_0(X)$. Let $N$ be the smallest integer such that $([L]-1)^N = 0$.
If $N=1$, then $[L]=1$ and  we get $L \cong O_X$.
If $N\geq 2$, we have $$1 = [L^{\otimes r}] = [L]^r = (1+([L]-1))^r = 1 +
r([L]-1) + \alpha([L]-1)^2 \ \in K_0(X),$$ for some $\alpha \in K_0(X)$.
Multiplying by $([L]-1)^{N-2}$ yields $$r([L]-1)^{N-1} = 0 \ \in K_0(X),$$ that
is,  $([L]-1)^{N-1}$ is a non-zero torsion element in $K_0(X)$.

It remains to see that if $S$ is a smooth projective surface such that $K_0(S)$
is torsion-free, then $\CH^2(S)$ is torsion-free. This follows immediately from
the fact \cite[Ex. 15.3.6]{fulton} that the second Chern class
    induces an isomorphism $c_2 : \mathrm{F}^2K_0(S)
\stackrel{\simeq}{\longrightarrow}
    \CH^2(S)$ (here, $\mathrm{F}^2K_0(S)$ is the subgroup
    of $K_0(S)$ spanned by coherent sheaves supported in codimension
    2). Since $K_0(S)$ is assumed to be torsion-free,
  we obtain that $\CH^2(S)$ is
    torsion-free.
\end{proof}

\begin{thm} \label{thm surface} Let $S$ be a smooth projective surface over a
field $k$, with $\chi(O_S)=1$. Assume that $S$ has a numerically exceptional collection of maximal length.
Then $\mathds{1} \oplus \mathds{1}(-1)^{\oplus \rho} \oplus \mathds{1}(-2)$ is a direct summand of the integral Chow motive of $S$. 
Moreover, if $S$ has a full exceptional collection, 
  then the integral Chow motive of $S$ is isomorphic to a sum of Lefschetz
  motives.
\end{thm}
\begin{proof} First assume that $S$ has a numerically exceptional collection $(E_0,\ldots,E_{n+1})$ of maximal length.
	By Proposition \ref{P:unimodularity},  there exists
	a $n$-tuple $(D_1,\ldots, D_n)$ of elements of $\CH^1(S)$, such that the matrix
	$M = (D_i\cdot D_j)_{1 \leq i,j \leq n}$ 
	is unimodular. (In particular, $\Z D_1 \oplus \cdots \oplus \Z D_n$ is a sub-group of $\CH^1(S)$.)
 Let
	then $(D_1^\vee,\ldots, D_n^\vee)$ be the basis of $\Z D_1 \oplus \cdots \oplus \Z D_n$ that is dual to
	the basis $(D_1,\ldots, D_n)$ with
	respect to the intersection pairing and let,  by Corollary \ref{cor zerocycle}, 
	$a \in \CH_0(S)$ be a zero-cycle of
	degree $1$ on $S$. We define the
	correspondences $\pi^0 := a\times S$, $\pi^4 := S \times a$,
	and $$p_i := D_i^\vee \times D_i, \quad \text{for all } 1\leq i
	\leq n.$$ These define mutually orthogonal idempotents in the
	correspondence ring $\CH^2(S\times S)$.  For instance $p_i \circ p_j
	= (D_j \cdot D_i^\vee)\, D_j^\vee \times D_i = \delta_{i,j}\, D_j^\vee \times
	D_i$, where
	$\delta_{i,j} = 0$ if $i\neq j$ and $\delta_{i,i}=1$. It is also clear that
	$\pi^0 \circ \pi^0 = \pi^0$, $\pi^4 \circ \pi^4 = \pi^4$, and that $\pi^0 \circ
	\pi^4 = \pi^4 \circ \pi^0 = \pi^4 \circ p_i = p_i \circ \pi^4 = \pi^0 \circ p_i
	= p_i \circ \pi^0 =0$. Moreover, we
	have $(S,\pi^0) \cong \mathds{1}$, $(S,\pi^4) \cong
	\mathds{1}(-2)$, and $(S,p_i) \cong \mathds{1}(-1)$ for all $i$, and this proves that $\mathds{1} \oplus \mathds{1}(-1)^{\oplus \rho} \oplus \mathds{1}(-2)$ is a direct summand of the integral Chow motive of $S$. 
	
	Assume now that the collection $(E_0,\ldots,E_{n+1})$ is  full exceptional.
	 It is enough to show that the idempotent correspondence $$\Gamma := \Delta_S
	- \pi^0 - \pi^4 - \sum_i p_i  \in\CH^2(S \times
	S)$$ is equal to $0$.
	The key is to show that $\Gamma$ acts as zero on the integral Chow groups
	$\CH^i(S)$ for $i=0,1$ and $2$.  
 For this, it is enough to show that $\operatorname{CH}^1(S) = 	\Z D_1 \oplus \cdots \oplus \Z D_n$ and that $\CH^2(S)=\Z a$.
	First recall that, in general, if $X$ is a smooth projective variety with a
	full exceptional
	collection that consists of $N$ exceptional objects, then $K_0(X)$ is free of rank $N$.
	The Chern character induces an isomorphism $K_0(S)\otimes_\Z \Q
	\stackrel{\simeq}{\longrightarrow} \CH^*(S) \otimes_\Z \Q$. Since $\rk \CH^0(S)
	= 1$ and since $\rk \CH^2(S) \geq 1$, we get that $\rk \CH^1(S) \leq n$.  
Since the intersection pairing restricted to the subspace of $\CH^1(S)$ spanned by $D_1,\ldots, D_n$ is unimodular, and since $\CH^1(S)$ is torsion-free by Lemma \ref{lem torsionfree}, we find that $(D_1,\ldots, D_n)$ forms a $\Z$-basis of $\CH^1(S)$.
	It follows that $\CH^2(S)$ has rank 1. Since $\CH^2(S)$ is also
	torsion-free by Lemma  \ref{lem torsionfree}, we find that $\CH^2(S)$
	is spanned by any zero-cycle of minimal positive degree and thus that
	$\CH^2(S)=\Z a$.

	Finally, by flat base-change \cite{kuznetsov}, the sequence $\langle
	(E_0)_K,(E_1)_K,\ldots, (E_{n+1})_K\rangle$ is a full exceptional
	collection for $S_K = S\times_{\mathrm{Spec}\, k}
	\mathrm{Spec}\, K$ for all field extensions $K/k$. Here $(E_i)_K$ is the
	pull-back of the object $E_i$ along the projection $S_K \rightarrow S$.
	Hence,
	the arguments above show that $(\Gamma_K)_*\CH^*(S_K) = 0$ for all
	field extensions $K/k$. Therefore the
	correspondence $\Gamma$ is nilpotent\,; see e.g. \cite[Proposition 3.2]{sv}.
	Because $\Gamma$ is an
	idempotent, we conclude that $\Gamma =0$.
\end{proof}

\begin{rmk}
One may in fact prove the following generalization of Theorem \ref{thm
surface}\,: If $S$ is a surface such that the base-change $K_0(S) \rightarrow
K_0(S_K)$ is surjective for all field extensions $K/k$, then the integral Chow
motive of $S$ is isomorphic to a direct sum of Lefschetz motives. (This is indeed a
generalization of Theorem \ref{thm surface} because of the base-change theorem
\cite{kuznetsov} for full exceptional collections.) For that matter, one uses a
recent result of Totaro \cite[Theorem 4.1]{totaro}, combined with the
integral version of the Riemann--Roch formula as used in the proof of Lemma
\ref{lem
torsionfree}.
\end{rmk}

\begin{rmk} \label{rmk barlow}
It seems that the converse to Theorem \ref{thm surface} is not true, that is,
for a surface to have a full exceptional collection seems more restrictive than
its integral Chow motive being isomorphic to a direct sum of Lefschetz motive.
Consider for instance a complex Barlow surface $S$. On the one hand, it is
proved in
\cite[Proposition 1.9]{actp} and \cite[Corollary 2.2]{voisin} that the Chow
group of zero-cycles of $S$ is universally trivial, in the sense that
$\CH^2(S_K)=\Z$
for all field extensions $K/\C$. By \cite[Theorem 4.1]{totaro},
 it follows
that the integral Chow motive of $S$ is a direct sum of Lefschetz motives.  On the
other hand, B\"ohning, Graf von Bothmer,
Katzarkov, and Sosna \cite{bgks} have exhibited a complex Barlow surface $S$ (a
determinantal Barlow surface) with an exceptional collection whose orthogonal
complement is a \emph{phantom} category, that is, a
non-trivial strictly full triangulated category with vanishing $K_0$.
Of course this does not say that the Barlow surface $S$ does not admit a full
exceptional collection, but it looks like a possibility that it won't. Finally,
as yet another reason why having a full exceptional collection is stronger than
having an integral motive isomorphic to a direct sum of Lefschetz motives, it
is believed and conjectured that a  surface that admits a full
exceptional collection must be rational.
\end{rmk}

\section{Numerically exceptional collections of maximal length on surfaces}

In the previous section, we saw  that the existence of a full
exceptional collection for a surface $S$ gives serious
constraints  on the integral motive of $S$. 
In this section, we show that, for a surface $S$,  the weaker condition of having a
numerically
exceptional collection  of maximal length, consisting of rank one objects,  is
still very restrictive.  The main result, which  builds up on Proposition \ref{prop
unimodular}, is Theorem \ref{thm criterion}\,:  we give a necessary and
sufficient
condition for a
smooth
projective surface $S$ defined over a field $k$, with $\chi(O_S)=1$, to
admit a numerically exceptional collection  of maximal length, consisting of line-bundles. Although its
proof  consists mostly of elementary linear algebra and lattice theory, Theorem
\ref{thm criterion} has surprising consequences. On an arithmetic perspective,
Theorem \ref{thm cycle} roughly says that a ``non-split'' surface over a field
$k$  (e.g. a surface that is not rational over $k$ but that becomes rational
after some field extension\,; see Theorem \ref{thm rational}) does not admit a
numerically exceptional collection
of maximal length, consisting of line-bundles. On a geometric perspective, Theorem \ref{thm surface general
type}  determines exactly which complex surfaces
with $p_g=q=0$ admit a numerically exceptional collection of maximal length.

\subsection{Main theorem}
Before we proceed to the statement of Theorem \ref{thm criterion}, let us recall
some facts about lattices for which we refer to \cite{hm}. A lattice $\Lambda$
is a free $\Z$-module of finite rank equipped with a symmetric bilinear form
$b : \Lambda \times \Lambda \rightarrow \Z$. A  lattice is said to be \emph{even} if
the norm of every vector is even\,; it is said to be \emph{odd} otherwise. A
lattice is said to be \emph{unimodular} if the determinant of its bilinear form
(expressed in any $\Z$-basis) is equal to $\pm 1$. An odd unimodular lattice of
signature $(1,N)$ is always isomorphic to the lattice $\langle -1
\rangle^{\oplus N} \oplus \langle 1 \rangle$. Here, $\langle \pm 1\rangle$
denotes the lattice of rank $1$ with generator of norm equal to $\pm 1$, and the
direct sum is understood as being orthogonal. 
An even unimodular lattice of
signature $(1,N)$ exists only when $N-1$ is divisible by $8$, in which case it
is isomorphic to $U\oplus E_8(-1)^{\oplus \frac{N-1}{8}}$. Here, $U$ is the
hyperbolic plane and $E_8(-1)$ is the opposite of the $E_8$-lattice.

\begin{thm} \label{thm criterion} Let $S$ be a smooth projective surface over a
field $k$, with $\chi(O_S)=1$. The
following statements are equivalent\,:
\begin{enumerate}[(i)]
\item $S$ admits a numerically
  exceptional collection $(L_0,L_1,\ldots, L_{n+1})$ of line-bundles which is
  of \emph{maximal length}, that is, $n = \rk \mathrm{N}^1(S)$.
\item  $S$ admits a numerically
  exceptional collection $(E_0,E_1,\ldots, E_{n+1})$ of numerical line-bundles in $D^b(S)$
which is
  of \emph{maximal length}, that is, $n = \rk \mathrm{N}^1(S)$.
\item We have $(K_S)^2=10-  \rk \mathrm{N}^1(S)$, and the lattice
$\mathrm{N}^1(S)$ 
  and the canonical divisor $K_S$, when seen as an element of
$\mathrm{N}^1(S)$, satisfy one of the following properties\,:
  \begin{itemize}
\item $\mathrm{N}^1(S)\cong \langle 1 \rangle$ and $K_S = 3D$ for some
primitive divisor $D$\,;
\item   $\mathrm{N}^1(S) \cong U$ and $K_S = 2D$ for some
primitive divisor $D$\,; 
\item  $\mathrm{N}^1(S) \cong \langle 1 \rangle \oplus \langle -1
\rangle^{\oplus n}$ with $n>0$ and $K_S$ is
a primitive divisor.
\end{itemize}
  \end{enumerate}
\end{thm}

A first step towards proving the theorem consists in characterizing the N\'eron--Severi lattice, together with the way the canonical divisor sits in it, of surfaces that admit a numerically exceptional collection of maximal length\,:

\begin{prop} \label{prop unimodular} Let $S$ be a smooth projective surface with $\chi(O_S)=1$. The following statements are equivalent.
	\begin{enumerate}[(i)]
		\item $S$ admits a numerically
		exceptional collection $(E_0,E_1,\ldots, E_{n+1})$ of numerical line-bundles, which is
		of \emph{maximal length}, that is, $n = \rk \mathrm{N}^1(S)$.
		\item The N\'eron--Severi lattice $\mathrm{N}^1(S)$ is trigonal and unimodular, and $K_S$ is a \emph{special} characteristic element in the sense of Definition \ref{D:characteristic}.
	\end{enumerate}	
\end{prop}
\begin{proof} By Proposition \ref{P:unimodular}, $(ii)$ is equivalent to the existence of $n+1$ divisors $D_1,\ldots,D_{n+1}$ in $\mathrm{N}^1(S)$ such that the matrix $(D_i\cdot D_j)_{1\leq i,j \leq n+1}$ is trigonal and such that $K_S\cdot D_i = -2-(D_i)^2$ for all $i$. This in turn is equivalent, by Proposition \ref{prop intersection matrix},  to the existence of a numerically exceptional collection of maximal length consisting of numerical line-bundles. 
\end{proof}

\begin{proof}[Proof of Theorem \ref{thm criterion}]
  $(i)\Rightarrow (ii)$\,: This is obvious.\medskip

 $(ii) \Rightarrow (iii)$\,: By Proposition \ref{prop unimodular}, the N\'eron--Severi lattice $\mathrm{N}^1(S)$ is trigonal and unimodular, and $K_S$ is a \emph{special} characteristic element in the sense of Definition \ref{D:characteristic}, \emph{i.e.}, $K_S\cdot c_1(E_i) = -c_1(E_i)^2 - 2$ for all $0\leq i \leq n+1$. We can conclude by invoking Theorem \ref{T:mainappendix} (and specifically the implication $(i) \Rightarrow(ii)$ therein).\medskip
    
$(iii) \Rightarrow (i)$\,:  This is the conjunction of Proposition \ref{prop unimodular}, and of the implication $(ii) \Rightarrow (i)$ of Theorem \ref{T:mainappendix}.
In fact, thanks to item $(iii)$ of Theorem \ref{T:mainappendix}, we can find an orthogonal basis of $\mathrm{N}^1(S)$ in which $K_S$ has a nice expression.
Assume that  $\mathrm{N}^1(S)$ is an odd
unimodular lattice of rank $n$. Since it has signature $(1,n-1)$ and since
$(K_S)^2=10-n$ by assumption, Theorem \ref{T:mainappendix} implies that there exists a $\Z$-basis  $(D_1, \ldots,
D_n)$ such
that $K_S = D_1 + D_2 + \cdots + D_{n-1}  - 3D_n$, $D_i\cdot D_j = 0$ for $i\neq
j$, $(D_i)^2 = -1$ for $i \leq n-1$ and $(D_{n})^2=1$. Let us then define
$D_{n+1} = 2D_n$. Then we easily check that 
the
collection $(O_S,O_S(D_1),\ldots,O_S(D_{n+1}))$ is numerically exceptional.
        Likewise, if  $\mathrm{N}^1(S)$ is isomorphic to the hyperbolic plane
and if $K_S$ is twice a primitive divisor, then, because $(K_S)^2=8$ by
assumption, Theorem \ref{T:mainappendix} gives
a $\Z$-basis $(D_1,D_2)$ of $\mathrm{N}^1(S)$ such that $(D_1)^2 =
(D_2)^2 = 0$, $D_1\cdot D_2 = 1$ and $K_S = -2D_1-2D_2$. We then define $D_3 :=
D_1+D_2$.
 Again it is straightforward to check that  
the collection
$(O_S,O_S(D_1),O_S(D_2),O_S(D_1+D_2))$ is numerically exceptional.
\end{proof}

\subsection{Consequence for the cycle class map} Let $k$ be a field and denote 
$\bar{k}$ a separable
closure. Given a field extension $K/k$ and a scheme $X$ over $k$,  we write $X_{K} := X\times_{\mathrm{Spec}\,
	k}{\mathrm{Spec} \,K}$.

Theorem \ref{thm criterion} gives constraints of
arithmetic nature for the existence of numerically exceptional collections, consisting of numerical line-bundles, of
maximal length\,:

\begin{thm} \label{thm cycle}
Let $S$ be a smooth projective surface over a
field $k$, with $\chi(O_S)=1$ and with first Betti number $b_1 = 0$. 
Assume that
$S$ admits a numerically
  exceptional collection $(E_0,E_1,\ldots, E_{n+1})$ of maximal length, consisting of numerical line-bundles. Then, for all primes $\ell$  not dividing $\mathrm{char} \, k$, the cycle
class
map $$\CH^1(S) \otimes \Z_\ell \rightarrow
\HH_{et}^2(S_{\bar{k}},\Z_\ell(1))$$ is
surjective modulo torsion, that is, it induces a surjective map $$\CH^1(S)
\otimes \Z_\ell \twoheadrightarrow
\HH_{et}^2(S_{\bar{k}},\Z_\ell(1))/torsion.$$
In particular,  the
 collection $((E_0)_K,(E_1)_K,\ldots, (E_{n+1})_K)$ for $S_K$ is numerically
    exceptional of maximal
length, and the base-change $$\mathrm{N}^1(S) \stackrel{\simeq}{\longrightarrow}
\mathrm{N}^1(S_K)$$ is an isometry for all field extensions $K/k$.
\end{thm}
In other words, in the same way a full exceptional collection remains full exceptional after extension of the base field \cite{kuznetsov}, a numerically exceptional collection of numerical line-bundles of maximal
length on a
surface $S$ with $p_g = q= 0$ remains of maximal length after any field
extension.
Note that, by Theorem \ref{thm surface}, if $\langle E_0,E_1,\ldots,
E_{n+1}\rangle$ is full exceptional, then the cycle class map $\CH^1(S)
\otimes \Z_\ell \rightarrow
\HH_{et}^2(S_{\bar{k}},\Z_\ell(1))$ is
surjective.

Note that surfaces $S$ with $\chi(O_S)=1$ and $b_1=0$ include surfaces with $p_g = q = 0$\,; see \emph{e.g.} \cite[\S 3.4]{liedtke}. These conditions are equivalent in characteristic zero, or if the surface lifts to characteristic zero. However, in positive characteristic, there are examples of surfaces with $b_1=0$ and $p_g=q > 0$, \emph{e.g.}, non-classical Godeaux surfaces \cite{liedtke2}.

\begin{proof}[Proof of Theorem \ref{thm cycle}]
By Noether's formula, $(K_S)^2 = 10- b_2$.
By Theorem \ref{thm criterion}, we also have $(K_S)^2 = 10-\rho$. This implies that $\rho = b_2$. On the other hand, Theorem \ref{thm criterion} also
 says that the
intersection pairing on $\mathrm{N}^1(S)$ is unimodular. This finishes the proof
of the theorem.
\end{proof}

\subsection{On a result of Hille and Perling \cite{hp}}
The aim of this paragraph is to extend the main result of Hille--Perling
\cite{hp} to surfaces that are defined over non-algebraically closed fields and
that admit numerically exceptional collections of maximal length (rather than
full exceptional) consisting of line-bundles.
Hille and Perling proved the following (we refer to \cite[Theorem 11.3]{perling}
for a precise statement)\,:

\begin{thm}[Hille--Perling \cite{hp}] \label{thm hp}
Let $S$ be a smooth projective surface defined over an algebraically closed
field $k$. Assume that $S$ admits a full exceptional collection  $(E_0,\ldots,
E_{n+1})$ consisting of line-bundles. Set $E_{n+2}:=E_0(-K_S)$.  Then to this
sequence there is associated in a canonical way a smooth complete toric surface
with torus invariant prime divisors $\Delta_0,\ldots, \Delta_{n+1}$ such that
$\Delta_i^2+2=\chi(E_{i+1}\otimes E_i^{-1})$ for all $0 \leq i \leq n+1$.
\end{thm}

In light of Theorem \ref{thm criterion}, the assumption that the base field $k$ is algebraically closed in Theorem \ref{thm hp} can be lifted\,:

\begin{thm}
Let $S$ be a smooth projective surface defined over a field $k$. Assume that
$\chi(O_S)=1$ and that $S$ admits a numerically exceptional collection
$(E_0,\ldots, E_{n+1})$ of maximal length, consisting of line-bundles. Then the
conclusion of Theorem \ref{thm hp} holds.
\end{thm}
\begin{proof} Let us define as before, for $1\leq i \leq n+1$, $D_i:=
c_1(E_i)-c_1(E_{i-1})$. Let us also define, following Hille and Perling \cite[p.
1242]{hp}, $D_{0} := -K_S-\sum_{i=1}^{n+1} D_i$ (compare with Proposition \ref{P:unimodular}). By convention, we set
$D_{i+n+2} :=D_i$. Then, by Proposition \ref{prop intersection matrix},
 we get

\begin{enumerate}[(i)]
\item $D_i\cdot D_{i+1} = 1$ for all $i$\,;
\item $D_i\cdot D_j = 0$ for $i\neq j$ and $\{i,j\}\neq\{l,l+1\}$ for all $0
\leq l \leq n+1$\,;
\item $\sum_{i=1}^{n+2} D_i = -K_S$.
\end{enumerate}
The data consisting of $\{D_i, 1\leq i\leq n+2\}$ will define an \emph{abstract
toric system} in the sense of \cite[Definition 2.6]{hp} if the extra condition
\begin{enumerate}[(i)]
\item[(iv)] $\sum_{i=1}^{n+2} (D_i)^2 = 12 - 3(n+2)$
\end{enumerate}
\noindent holds. Conditions $(i)$, $(ii)$ and $(iii)$ yield $(K_S)^2 = 2(n+2) + \sum_{i=1}^{n+2}
(D_i)^2$. But then, by our main Theorem \ref{thm criterion}, we have
$(K_S)^2=10-n$. Therefore (iv) does indeed hold, so that  $\{D_i, 1\leq i\leq
n\}$ does define an abstract toric system. 
The theorem then follows because the proof of \cite[Theorem 3.5]{hp} depends
only on the combinatorial data of an abstract toric system.
\end{proof}

\subsection{Exceptional collections and rational surfaces} 
The following theorem is due to Manin \cite{manin} and Iskovskikh \cite{isk}\,;
see also \cite[Theorem 3.9]{hassett}.
\begin{thm} \label{thm geomrat surfaces} Let $S$ be a smooth projective minimal
surface defined over a
perfect field $k$. Assume that $S_{\bar k}$ is rational. Then $S$ is one of the
following\,:
\begin{itemize}
\item $\PP^2$\,;
\item $S\subset \PP^3$ a smooth quadric with $\mathrm{Pic}(S)=\Z$\,;
\item a del Pezzo surface with $\mathrm{Pic}(S)=\Z K_S$\,;
\item a conic bundle $f:S\rightarrow C$ over a conic, with $\mathrm{Pic}(S)
\cong \Z \oplus \Z$.
\end{itemize}
\end{thm}

The following theorem is a consequence of Theorem \ref{thm criterion}\,; it
shows that geometrically rational, but non-rational, minimal surfaces defined over a
perfect field do not admit an exceptional collection of maximal length.
\begin{thm} \label{thm rational}
Let $S$ be a geometrically rational, smooth projective surface defined over a perfect field
$k$ that admits a
numerically exceptional collection $(E_0,\ldots, E_{n+1})$ of maximal length. Assume either that $S$ is minimal, or that the objects $E_i$ are numerical line-bundles. Then $S$ is rational.
\end{thm}

\begin{lem} \label{L:min-unimodular}
Let $S$ be a smooth projective surface over a perfect
field $k$, and denote $\Sigma$ a minimal model of $S$. Assume that the N\'eron--Severi lattice $\operatorname{N}^1(S)$ is unimodular. Then $S$ is obtained from $\Sigma$ by successively blowing up rational $k$-points on $\Sigma$. Moreover, $\operatorname{N}^1(\Sigma)$ is unimodular.
\end{lem}
\begin{proof}
First note that $S$ is obtained from
$\Sigma$ by
successively blowing up Galois-invariant closed points\,; see for instance
\cite{hassett}.
Let $\tilde{T}$ be the blow-up of a smooth projective surface $T$ over $k$ along a Galois-invariant closed point of degree $d\geq 1$. Such a blow-up produces a Galois-invariant
collection of pairwise disjoint $(-1)$-curves, say $E_1,\ldots, E_d$,
and the N\'eron--Severi lattice of $\tilde{T}$ splits orthogonally as $\langle E \rangle
\oplus N'$, where $E=E_1+\cdots + E_d$ and $E^2=-d$, for some lattice $N'$.
This establishes the lemma.
\end{proof}

\begin{proof}[Proof of Theorem \ref{thm rational}] 
Let $S$ be a geometrically rational, smooth projective surface defined over a perfect field
$k$ that admits a
numerically exceptional collection  of maximal length. By Proposition \ref{P:unimodularity}, its N\'eron--Severi lattice is unimodular. It follows from Lemma \ref{L:min-unimodular} that $S$ is obtained from one of the minimal surfaces
listed in Theorem \ref{thm geomrat surfaces} by successively blowing up rational
points. Denote $\Sigma$ a minimal model for $S$\,; $\operatorname{N}^1(\Sigma)$ is unimodular.

If $\Sigma = \PP^2$, then $S$ is obviously rational.

If $\Sigma$  is a smooth quadric in $\PP^3$ with
$\Pic(\Sigma)=\Z$, then $\rho(\Sigma) \neq b_2(\Sigma) = 2$ modulo $8$. It follows that $\rho(S) \neq b_2(S)$ modulo $8$. Therefore, by Proposition \ref{P:unimodularity}, $S$ does not admit a numerically exceptional collection of maximal length and we get a contradiction.

If $\Sigma$ is a conic bundle $f:\Sigma\rightarrow C$ over a conic, with $\mathrm{Pic}(\Sigma)
\cong \Z \oplus \Z$, then by the unimodularity $\Sigma$ has a degree-one zero-cycle and hence $C = \PP^1$.   Note that $\operatorname{N}^1(\Sigma)$ is spanned by a fiber $F$ and by a multi-section $D$. 
Indeed $\operatorname{N}^1(\Sigma)$  cannot be spanned by two vertical components, \emph{i.e.}, by irreducible components of some fibers, since otherwise the intersection pairing on $\operatorname{N}^1(\Sigma)$ would be negative, contradicting the Hodge index theorem. Suppose now that $\operatorname{N}^1(\Sigma)$ is spanned by  two multi-sections $D$ and $D'$. Then there exist co-prime integers $u$ and $v$ such that $uD + vD'$ is torsion in $\CH^1(\Sigma_\eta)$, where $\Sigma_\eta$ is the conic that is the generic fiber of $f$. Thus, by localization for Chow groups, we see that $\operatorname{N}^1(\Sigma)$ is spanned by $D$ and by a vertical component. But since  $\operatorname{N}^1(\Sigma)$ has rank two we see \cite[\S 3.2]{hassett} that in fact $\operatorname{N}^1(\Sigma)$ is spanned by $D$ and a fiber $F$, as claimed. Since $F^2=0$, the unimodularity yields $D\cdot F =	1$, and hence that $D$ is in fact a section. We conclude by showing that this implies that $f$ is a smooth $\PP^1$-bundle\,: let $F'$ be a fiber of $f$ and denote $k_1/k$ the field of definition of $F'$.
Since $D\cdot F' = 1$,  we see that $F'$ has a $k_1$-rational point and that $F'$ is smooth (since otherwise the two geometric components of $F'$ would be defined over $k_1$ and thus would not be in the same Galois orbit and hence $\operatorname{N}^1(\Sigma)$ would have rank $\geq 3$). This implies that $F' = \PP^1_{k_1}$. This
proves that $\Sigma$ is a $\PP^1$-bundle over $\PP^1$, and hence that $S$ is rational.

If $\Sigma$ is a del
Pezzo surface with $\Pic(\Sigma)=\Z K_\Sigma$, then we distinguish between two cases. 
Assume $S$ admits a numerically exceptional collection of maximal length, consisting of numerical line-bundles. Note that if a surface $S$ is obtained from a surface $\Sigma$ by successively blowing up $k$-rational points, then, since each such blowup increases $\rk(\operatorname{N}^1)$ by $1$ and decreases $K^2$ by $1$, $(K_S)^2=10-\rk \mathrm{N}^1(S)$ if and only if $(K_{\Sigma})^2=10 - \rk
\mathrm{N}^1(\Sigma)$. Therefore, by Theorem \ref{thm criterion}, we have on the one hand that $\operatorname{N}^1(\Sigma)$ is unimodular and hence $K_{\Sigma}^2 = 1$, and on the other hand that $K_\Sigma^2 = 10 - \rk \operatorname{N}^1(\Sigma) = 9$, thus yielding a contradiction. 
Assume now that $S = \Sigma$ (\emph{i.e.}, $S$ is minimal). Then Theorem \ref{thm perling} gives a numerically exceptional collection of maximal length, consisting of $3$ numerical line-bundles, for $S$. We conclude to a contradiction as before.
\end{proof}

\begin{rmk}\label{R:rational}
It should be noted that the converse to Theorem \ref{thm rational} does not
hold\,: a rational surface over a field $k$ need not admit a numerically
exceptional collection of maximal length. Consider for example the rational
surface $X$ defined over a
non-algebraically closed field obtained as the blow-up of the projective plane
$\PP^2$ along  a non-rational Galois-invariant closed point. Then the
N\'eron--Severi lattice of $X$ is not unimodular and, by virtue of Theorem
\ref{thm criterion}, $X$ does not admit a numerically exceptional collection of
maximal length.

In fact, a \emph{minimal} rational surface over a field $k$ need not admit a numerically exceptional collection of maximal length. Indeed, a smooth quadric $Q$ with a rational point is rational, but, if $\operatorname{Pic}(Q) = \mathbb{Z}$, then, by Proposition \ref{P:unimodularity}, $Q$ does not admit a numerically exceptional collection of maximal length since $\rho = b_2 -1$.
\end{rmk}

\subsection{Numerically exceptional collections of maximal length for complex
surfaces} \label{sec cpxsurface}
Let us now turn to geometric consequences. Until the end of this paragraph, the
base field is the field of complex numbers. In Theorem \ref{thm surface general
type}, we determine
exactly which smooth projective complex surfaces with $p_g=q=0$ admit a
numerically exceptional collection of maximal length. \medskip

 Let $S$ be a smooth minimal projective complex surface with $q = p_g =0$. Then,
by the Enriques--Kodaira classification of compact complex surfaces according to
their Kodaira dimension $\kappa$, $S$ is one of the following (see e.g.
\cite{kotschick})\,:

$\bullet \ \kappa = -\infty$, a minimal rational surface.  The minimal rational
surfaces are $\PP^2$
 and the Hirzebruch surfaces $\Sigma_n$, $n =
0,2,3,4,\ldots$, where $\Sigma_n$ is the $\PP^1$-bundle $\PP(O\oplus O(-n))$
over
$\PP^1$. For instance, $\Sigma_0 = \PP^1 \times \PP^1$. Note that $\Sigma_1$ is
not minimal, it is $\PP^2$ blown up once. Note also that $K^2 = 9$ for $\PP^2$
and $K^2 = 8$ for $\Sigma_n$.

$\bullet \ \kappa=0$ or $1$, a minimal properly elliptic surface. Let $X_9$ be the
rational elliptic surface obtained from $\PP^2$
 by blowing up the
nine base points of a generic cubic pencil. Then Dolgachev \cite{dolgachev}
proved that the minimal complex
surfaces with $p_g=q=0$
 of Kodaira dimension $0$ or $1$ are  obtained from $X_9$
by performing logarithmic
transformations on at least two different smooth fibers. We denote the surfaces
obtained in this way by $X_9(p_1,... ,p_n)$ ($p_1 \leq p_2 \leq \ldots \leq
p_n$), where the $p_i$ are the
multiplicities of the
logarithmic transformations, and call them Dolgachev surfaces. (Some authors
reserve
this name for the case when there are only two multiple fibers and their
multiplicities
are relatively prime.)  Now $X_9(2,2)$ is the Enriques surface. It is
the only Dolgachev surface with Kodaira dimension $0$. Note that $K_S^2 =
0$ for all of these surfaces.

$\bullet \ \kappa=2$, a minimal surface of general type.
 For a minimal surface of general type, we have $K_S^2 > 0$, as well as Castelnuovo's inequality $c_2 > 0$. If in
addition
$q = p_g = 0$, then $K_S^2 + c_2 = 12$ by Noether's formula, and $K_S^2 \le 9$.\medskip

In fact, 
unless $S$ has Kodaira dimension $=0$ or $1$, there are no
obstructions to the existence of a numerically exceptional collection of maximal
length\,:

\begin{thm}\label{thm surface general type} Let $S$ be a
	smooth projective complex surface  with
	$p_g=q=0$. 
	\begin{itemize}
		\item If $S$ is not minimal, then $S$ has a numerically exceptional collection
		of maximal length.
	\end{itemize}
	Assume now that $S$ is minimal.
	\begin{itemize}
		\item If $\kappa(S) = -\infty$, then $S$ has a numerically
		exceptional collection of maximal length.
		\item If $\kappa(S) = 0$, then $S$ is an Enriques surface and
		it does not have a
		numerically exceptional collection of maximal length.
		\item If $\kappa(S)=1$, then $S$ is a Dolgachev surface
		$X_9(p_1,\ldots,p_n)$, and $S$ has a numerically exceptional collection of
		maximal length if and only if $S$ is one of $X_9(2,3)$, $X_9(2,4)$, $X_9(3,3)$,
		$X_9(2,2,2)$. 
		\item If  $\kappa(S) = 2$, then $S$ has a numerically
		exceptional collection of maximal length.
	\end{itemize}
\end{thm}
\begin{proof} First note that under the condition $p_g=q=0$, we have $\chi(O_S)=1$, $b_2 = \rho$, and the intersection pairing on $\operatorname{N}^1(S)$ is unimodular. Indeed, the first Chern
class  induces an isomorphism $\mathrm{Pic}(S) \cong \HH^2(S,\Z)$, so that
$\mathrm{N}^1(S)
\cong
\HH^2(S,\Z)/torsion$\,;  furthermore, by Poincar\'e duality, it follows that the intersection
pairing on
$\mathrm{N}^1(S)$ is unimodular. 

If $S$ is not minimal, that is if $S$
is the blow-up of a smooth projective surface, then $\mathrm{N}^1(S)$ is
odd of rank $\geq 2$ and $K_S$ is clearly primitive. Hence, by Theorem \ref{thm
criterion}, if $S$ is not minimal, then $S$ admits a numerically exceptional
collection of line-bundles of
maximal length. 

From now on, we assume that $S$ is a minimal surface. By Theorem \ref{thm perling} and Theorem \ref{thm criterion}, note that, since $b_2 = \rho$, $S$ has a numerically exceptional collection of maximal length if and only if it has one consisting of line-bundles.
\medskip

$\bullet$ If $\kappa = -\infty$, then in fact $S$ has a full exceptional
collection of
line-bundles\,: if $S=\PP^2$, then the Beilinson collection $\langle
O_S,O_S(1),O_S(2)\rangle$
is full exceptional\,; if $S=\Sigma_n$, $n=0,2,4,\ldots$, is a Hirzebruch surface,
then denoting respectively $F$ and $C$ a fiber and a section  of the
corresponding $\PP^1$-bundle (so that $F^2=0, F\cdot C=1$ and $C^2=-n$), we have
that
 $\langle O_S,O_S(F),O_S(C),O_S(C+F)\rangle$
  is a full exceptional collection (this is essentially contained in
\cite{orlov}).
\medskip

$\bullet$ If $\kappa = 2$, then we know that $K_S^2 \in \{1,2,\ldots,9\}$.
Thus if
$K_S=rD$ for some positive integer $r$ and some primitive divisor $D$, then $r
\in \{1,2,3\}$.
Also, by Noether's formula, the N\'eron--Severi lattice $\mathrm{N}^1(S)$ has
rank $\in \{1,2,\ldots,9\}$, and, by the classification of unimodular lattices
of signature $(1,n-1)$, we see that $\mathrm{N}^1(S)$ is even if and
only if it is isomorphic to the hyperbolic plane $U$. Now, we have
\begin{lem} \label{lem even-odd}
Let $S$ be a complex surface with $p_g=q=0$. Then, the N\'eron--Severi lattice
$\mathrm{N}^1(S)$ is even if and only if $K_S = 2D$ for some 
divisor $D\in \mathrm{N}^1(S)$.
\end{lem}
\begin{proof}[Proof of the lemma]
Recall that $K_S$ is a characteristic element in $\operatorname{N}^1(S)$, that is, 
$E\cdot(E-K_S)$ is even for all $E \in \mathrm{N}^1(S)$ (this goes by the name of Wu's formula\,; it follows from the Riemann--Roch formula \eqref{RR} whereby $E\cdot(E-K_S) = 2(\chi(O_S(E)) - \chi(O_S))$).
Therefore $E^2$ is even for
all $E \in \mathrm{N}^1(S)$ if and only if $K_S \cdot E$ is even for all $E \in
\mathrm{N}^1(S)$. Thus, if $K_S = 2D$, then $\mathrm{N}^1(S)$ is even.
Conversely, the pairing on $\mathrm{N}^1(S)$ is
unimodular by Poincar\'e duality, and hence induces an isomorphism from $\mathrm{N}^1(S)$ to its dual. Since $K_S$ is characteristic and $\mathrm{N}^1(S)$ is assumed to be even, the element $K_S \in \mathrm{N}^1(S)$ is mapped to $2w$, for some $w \in \mathrm{N}^1(S)^\vee$, under this isomorphism. It is then apparent that $K_S = 2D$ for some 
divisor $D\in \mathrm{N}^1(S)$.
\end{proof} 
By Lemma \ref{lem even-odd}, if the intersection pairing on $\mathrm{N}^1(S)$ is
odd, then $K_S$
is either primitive or equal to $3D$ for some primitive divisor $D$. In order to
conclude, we need to show that $K_S=3D$ if and only if $n=1$, which is further
equivalent by Noether's formula to $K_S^2=9$. Clearly if $K_S=3D$, then
$K_S^2=9D^2$ so that $K_S^2$ must be equal to $9$. If now $n=1$, then
$\mathrm{N}^1(S) = \Z H$ for some divisor $H$, which by Poincar\'e
duality satisfies $H^2 = 1$. The canonical divisor $K_S$ is then equal to $aH$
for some integer $a$. Since $K_S^2=9$, we find that $a=\pm 3$ and we are done.
By Theorem \ref{thm criterion}, we deduce that every minimal smooth projective
complex surface of general type with $p_g=q=0$ admits a numerically exceptional
collection of maximal length. 
\medskip

$\bullet$ If $\kappa=0$ or $1$, then $S$ is a Dolgachev surface. The N\'eron--Severi lattice of a Dolgachev
surface $S$
has rank 10, so that by Theorem \ref{thm criterion} (combined with Lemma
\ref{lem even-odd}) $S$ admits a numerically
exceptional collection  of maximal length if and only if $K_S$ is
primitive. 

Let us first consider the case $\kappa = 0$, that is, the case where $S$ is a classical Enriques surface. It is known that the canonical sheaf $\omega_S$ is $2$-torsion, and hence that $K_S = 0$ in $\mathrm{N}^1(S)$.  Therefore,  $S$ does not admit a numerically exceptional
collection of maximal length. (Recall also that the N\'eron--Severi lattice of
an Enriques surface is $U\oplus E_8(-1)$\,; it is even of rank $10$.)

Consider now a Dolgachev surface $S=X_9(p_1,\ldots,p_n)$ of Kodaira dimension $1$ (\emph{i.e.}, which is not $X_9(2,2)$). Its canonical
divisor $K_S$ is not torsion and is given by \cite[p. 129]{dolgachev} 
$$K_S = (n-1)F - \sum_{i=1}^n F_i \
\in \Pic(S),$$
where $F$ is the class of a general fiber and $F_i$ is the class of the multiple
fiber corresponding to $p_i$ (in particular, $p_iF_i = F \in \Pic(S)$). The
canonical divisor $K_S$
 may or may not be primitive.
First we show by elementary arithmetic that  if $S$ is not one of $X_9(2,3)$,
$X_9(2,4)$, $X_9(3,3)$, $X_9(2,2,2)$, then $K_S$ is not primitive. Let us assume
that $S$ is $X_9(p_1,\ldots,p_n)$ with $2\leq p_1\leq p_2 \leq \ldots \leq p_n$
and with distinct multiple fibers $F_i$ such that $p_iF_i =F$. Let us write $c:= \gcd(p_1,p_2)$, $d:=
\frac{p_1p_2}{\gcd(p_1,p_2)}$, and let $u$ and $v$ be integers such that
$up_1+vp_2=c$. Consider $G := vF_1 + uF_2 \in \mathrm{Pic}(S)$\,; then
note that $$dG =  \frac{vp_1p_2}{c}F_1 +
\frac{up_1p_2}{c}F_2 = \frac{vp_2}{c}F +
\frac{up_1}{c}F = F.$$
Note also
that, in the N\'eron--Severi lattice $\mathrm{N}^1(S)$, the fibers $F$, $F_1, F_2$ are integral multiples of $G$ (namely, $F=dG$, $F_1 = q_2G$ and $F_2 = q_1G$, where $p_1 = c q_1$ and $p_2 = c q_2$) and the fibers $F_3,
\ldots, F_n$ are all rational multiples of $G$.\medskip

If $n=2$, we claim that $K_S$ is primitive only if $(p_1,p_2)$ is one of $(2,3)$, $(2,4)$ or $(3,3)$. 
In that case, we have
\begin{equation}\label{eq KG}
K_S = F - F_1 - F_2 = (d-q_2-q_1)G =(c q_1q_2 - q_1 - q_2)G.
\end{equation}
Note that $c q_1q_2 - q_1 - q_2 = (c-1)q_1q_2 + (q_1 - 1)(q_2 - 1) - 1$.
Assume this is equal to $1$.
If $q_1 = 1$ then $(c-1)q_2 = 2$, hence $c = 2, q_2 = 2$\,; or $c = 3, q_2 = 1$.
The first gives $(2,4)$, the second gives $(3,3)$.
If $q_1,q_2 \ge 2$, then the second summand is $\ge 1$, hence the first is $\le 1$, hence $c = 1$,
hence $(q_1 - 1)(q_2 - 1) = 2$. This gives $(2,3)$.\medskip

If $n>2$, then $K_S$ is primitive only if $n=3$ and $(p_1,p_2,p_3) = (2,2,2)$. Indeed, we have  
\begin{align*}
K_S &= (2F-F_1-F_2) + (F-F_3) +
\cdots (F-F_{n-1}) - F_n\\
&= (2d-q_1-q_2)G + (F-F_3) +
\cdots (F-F_{n-1}) - F_n.
\end{align*}
On the one hand, we have (note that $q_1\le q_2$) $$2d-q_1-q_2 = 2p_1q_2 - q_1  - q_2 = (2p_1-1)q_2 -q_1 \geq 3q_2-q_1 \geq 2q_2$$ with equality if and only if $p_1=p_2=2$.
On the other hand, each divisor $F-F_i$ is a positive rational multiple of $G$. Together with the inequality $(2d-q_1-q_2) \geq 2q_2 \ge 2$, it
follows that $$K_S \geq 2G - F_{n},$$ and equality holds only if $n=3$ and $p_1=p_2=2$.
If we can write $K_S = \lambda G$ in $\mathrm{N}^1(S)$
for some rational
number $\lambda>1$, then $K_S$ cannot be primitive.
We deduce that $K_S$ can only be primitive when $n=3$,
$p_1=p_2=2$ and $F_3 = G$, that is, when $S = X_9(2,2,2)$.
\medskip
 
Finally we check that $K_S$ is primitive for the Dolgachev surfaces $X_9(2,3)$,
$X_9(2,4)$, $X_9(3,3)$ and $X_9(2,2,2)$. In the first three cases, by
\eqref{eq KG}, we have  $K_S=G$. More generally, when $n=2$, we claim that the divisor $G$
is primitive. We proceed as in \cite[p. 384]{bhp}, by contradiction. If, for
some rational number $0<\lambda <1$, the class $\lambda G$ is represented by a
divisor, then by the Riemann-Roch formula, either $\lambda G$ or $K_S-\lambda G$
is effective. Note then that an effective divisor $D$ such that $\deg(D\cdot F)
= 0$ is linearly equivalent to $a_1F_1 + a_2F_2$ for some non-negative integers
$a_1$ and $a_2$. Since clearly
 $\lambda G$ is not effective, $K_S-\lambda G$
must be effective, that is, we can write $K_S-\lambda G = a_1F_1+a_2F_2$ for
some  non-negative integers $a_1$ and $a_2$.
But then we obtain $$(d-q_1-q_2 -\lambda )  G = a_1F_1 + a_2F_2 =
(a_1q_2+a_2q_1)G$$ and hence we find that
$\lambda$ is an integer, which gives a contradiction. 

In the last case ($S=X_9(2,2,2)$), we have $K_S = 2F-F_1-F_2-F_3 = F_1+F_2-F_3$.
Again, if, for some rational number  $0<\lambda <1$, the class $\lambda K_S$ is
represented by a divisor, then by the Riemann--Roch formula, either $\lambda
K_S$ or $(1-\lambda)K_S$ is effective. Assume that  $\mu K_S$ is effective for
some $0<\mu < 1$. Then there exist non-negative integers $a_1,a_2,a_3$ such that
 $\mu K_S = a_1F_1+a_2F_2+a_3F_3$. Since $F_1$, $F_2$ and $F_3$ are numerically
equivalent, we find that $\mu = a_1 + a_2 + a_3$, in particular we find that
$\mu$ is an integer.
Theorem \ref{thm surface general type}  is now
proved. 
\end{proof}

\begin{rmk}\label{R:blowup}
Although an Enriques surface does not admit an exceptional collection of maximal
length, it would be very interesting to decide whether or not an Enriques
surface blown up at a point admits an exceptional collection of maximal
length. This would give an example of a triangulated category with an
exceptional collection of maximal length that admits an exceptional object whose
orthogonal complement does not admit an exceptional collection of maximal
length. Indeed, denoting $p : \tilde{S} \rightarrow S$ the blow-up of $S$ along
a point $P$ and denoting $E$ the exceptional divisor, we have by Orlov's blow-up formula a semi-orthogonal decomposition
$D^b(\tilde{S}) \cong \langle O_E(-1),p^*D^b(S)\rangle$. Then the right-orthogonal complement of the exceptional object $O_E(-1)$ in $D^b(\tilde{S})$ does not admit an exceptional collection of maximal length by Theorem
\ref{thm criterion}. This is related to the Jordan--H\"older property for
derived categories\,; \emph{cf.} \cite{kuz}.

It would also be interesting to exhibit  exceptional collections of maximal
length for the Dolgachev surfaces $X_9(2,3)$, $X_9(2,4)$, $X_9(3,3)$ and
$X_9(2,2,2)$. The orthogonal of such collections would yield new examples of
quasi-phantom categories (triangulated categories with torsion $K_0$) in the
case of the Dolgachev surfaces $X_9(2,4)$, $X_9(3,3)$ and
$X_9(2,2,2)$. In the case of $X_9(2,3)$, it would yield (if one believes in
Orlov's conjecture) a new example of phantom category (a non-zero triangulated
category with vanishing $K_0$). 

\emph{N.B.} Cho and Lee \cite{cholee} have recently constructed exceptional collections on some Dolgachev surfaces of type $X_9(2,3)$ of maximal length whose orthogonal complements provide examples of phantom categories.
\end{rmk}


\appendix

\section{On trigonal unimodular lattices}

The main result is Theorem \ref{T:mainappendix}. The equivalence $(i) \Leftrightarrow (ii)$ therein reduces the equivalence $(i) \Leftrightarrow (iii)$ of Theorem \ref{thm criterion} to a purely linear algebraic statement.\medskip

We refer to \cite{hm} for the basics of lattice theory.
A lattice $(\Lambda,b)$
is a free $\Z$-module $\Lambda$ of finite rank equipped with a symmetric bilinear form
$b : \Lambda \times \Lambda \rightarrow \Z$. The \emph{norm} of a vector $x\in \Lambda$ is $ b(x,x) \in \Z$. 
We denote $\langle a \rangle$ the rank-one lattice $\Lambda = \Z \lambda$ such that $b(\lambda,\lambda) = a$. The \emph{signature} of a lattice $(\Lambda, b)$ is $(n^+,n^-,n^0)$ if $\Lambda \otimes_\Z \R$ splits as the orthogonal sum $\langle 1 \rangle^{\oplus n^+} \oplus \langle -1 \rangle^{\oplus n^-} \oplus \langle 0 \rangle^{\oplus n^0}$. If $n^0 = 0$, that is, if $\Lambda$ is non-degenerate, we will omit the term $n^0$ from the signature of $\Lambda$. 

A lattice is said to be \emph{even} if
the norm of every vector is even\,; it is said to be \emph{odd} otherwise. A
lattice is said to be \emph{unimodular} if the determinant of its bilinear form
(expressed in any $\Z$-basis) is equal to $\pm 1$. 
Let us denote $U$ the hyperbolic lattice, that is the lattice $\Z e_1 \oplus \Z e_2$ with $b(e_1,e_2)=1$ and $b(e_i,e_i)=0$ for $i=1,2$\,; it is up to isomorphism the only even unimodular lattice of rank $2$.

\begin{defn}\label{D:trigonal}
We will say that a lattice $(\Lambda,b)$ of rank $n$ is \emph{trigonal} if there exists a $\Z$-basis $(e_1,\ldots, e_n)$ of $\Lambda$ such that the matrix of $b$ expressed in that basis has the trigonal form 

\begin{equation} \label{eq matrix}
M := \operatorname{Mat}_{(e_1,\ldots,e_n)}(b) =  \left( \begin{array}{cccccc}
a_1 & 1 & & & &  \\
1 & a_2 & 1 & & &  \\
& 1 & a_3 & \ddots & &\\
& & \ddots & \ddots & 1\\
& & &  1 & a_n \end{array} \right)
\end{equation}
where the entries outside the $3$ diagonals consist solely of zeroes.
For simplicity, we will write $$M = [a_1,a_2,\ldots,a_n]$$ and sometimes $M = \operatorname{trig}(a_1,a_2,\ldots,a_n)$, for clarity.
A $\Z$-basis in which $b$ takes a trigonal form will be called a \emph{trigonal basis} for $b$.
Such a trigonal reduction is fairly special for unimodular matrices\,; see Remark \ref{rmk K2}, but also
\cite{newman} where it is shown that for any unimodular  bilinear form over $\mathbb Z$ there is a basis in which its matrix takes the form \eqref{eq matrix} with the $(n-1,n)$ and $(n,n-1)$ entries replaced by some positive integer $d$ .

\end{defn}

\begin{defn} \label{D:characteristic}
	Recall that an element $\omega$ in a lattice $(\Lambda,b)$ is said to be \emph{characteristic} if $b(\omega,\lambda) = b(\lambda,\lambda)$ (mod $2$) for all $\lambda \in \Lambda$. 
	We will say that a characteristic element $\omega$ in a trigonal lattice $(\Lambda,b)$ is  \emph{special} if there exists a trigonal basis $(e_1,\ldots,e_n)$ of $(\Lambda,b)$ such that $$b(\omega,e_i) = -b(e_i,e_i)-2, \quad \mbox{for all} \ 1\leq i \leq n.$$
	(Note that such a special characteristic element always exists if the trigonal lattice $\Lambda$ is unimodular.)
\end{defn}

First we characterize trigonal unimodular lattices  (endowed with a special characteristic element)\,:

\begin{prop} \label{P:unimodular} Let $(\Lambda,b)$ be a lattice of rank $n$ and signature $(n^+,n^-,n^0)$. The following statements are equivalent.
	\begin{enumerate}[(i)]
		\item $(\Lambda,b)$ is trigonal and unimodular.
			\item There exist $\lambda_1, \ldots, \lambda_{n+1}$ in $\Lambda$ such that 
			$\big(b(\lambda_i,\lambda_j)\big)_{1\leq i,j \leq n+1}$ is a trigonal matrix.
		\item There exist $\lambda_0, \lambda_1, \ldots, \lambda_{n+1}$ in $\Lambda$ such that 
		$$\big(b(\lambda_i,\lambda_j)\big)_{0\leq i,j \leq n+1} =  \left( \begin{array}{cccccc}
		a_0 & 1 & & &  (-1)^{n^+-1}  \\
		1 & a_1 & 1 & & &  \\
		& 1 & a_2 & \ddots & &\\
		& & \ddots & \ddots & 1\\
		(-1)^{n^+-1} & & &  1 & a_{n+1} \end{array} \right).$$
	
	\end{enumerate}
Moreover, assuming $(\Lambda,b)$ is trigonal and unimodular, a characteristic element $\omega \in \Lambda$ is special if and only if there exist $\lambda_0, \lambda_1, \ldots, \lambda_{n+1}$ in $\Lambda$ as in $(iii)$ with the additional property that $\omega = -\sum_{i=0}^{n+1} \lambda_i$.
\end{prop}
\begin{proof} $(i) \Rightarrow (iii)$\,: Let $(e_1,\ldots, e_n)$ be a trigonal basis for $(\Lambda,b)$. Since $(\Lambda,b)$ is assumed to be unimodular, $b$ identifies naturally $\Lambda$ with its dual. We define $e_0$ (resp. $e_{n+1}$) to be the dual of $e_1$ (resp. $e_n$), that is, $e_0$ (resp. $e_{n+1}$) is the element of $\Lambda$ such that $b(e_0,e_i) = 1$ if $i=1$ and $0$ otherwise (resp. such that $b(e_0,e_i) = 1$ if $i=n$ and $0$ otherwise).  We claim that the matrix $\big(b(e_i,e_j)\big)_{0\leq i,j \leq n+1}$ is as in $(iii)$. It is enough to check that $b(e_0,e_{n+1}) = (-1)^{n^+-1}$. We have $b(e_0,e_{n+1}) = b^{-1}(e_1,e_n)$, where $b^{-1}$ is the symmetric bilinear form on $\Lambda$ whose matrix expressed in the basis $(e_1,\ldots, e_n)$ is the inverse of  $M := \big(b(e_i,e_j)\big)_{1\leq i,j \leq n}$. Thus, denoting $m_{1,n}$ the $(1,n)^{\mathrm{th}}$ minor of $M$ (that is,  the determinant of the submatrix formed by deleting the $1^{\mathrm{st}}$ row and $n^{\mathrm{th}}$ column of $M$), we have  $$b(e_0,e_{n+1}) = (-1)^{n+1} (\det M)^{-1} m_{1,n} = (-1)^{n+1}(-1)^{n^-} = (-1)^{n^+-1},$$
where in the second equality we have used that $\det M = (-1)^{n^-}$ and $	m_{1,n} = 1$.
	\medskip
	
	$(iii) \Rightarrow (ii)$\,: This is obvious. \medskip
	
$(ii) \Rightarrow (i)$\,: We are going to show  that the determinant of $\big(b(\lambda_i,\lambda_j)\big)_{1\leq i,j \leq n} = [a_1,\ldots, a_n]$ is equal to $(-1)^{n^-}$. Since $\Lambda$ has rank $n$, this will prove that $(\Lambda,b)$ is unimodular and that $(\lambda_1,\ldots, \lambda_n)$ provides a trigonal basis of $\Lambda$.
Let us consider, for $m\leq n+1$, the $(m\times m)$-trigonal matrix
	$ [a_1,\ldots, a_m],$
	and let us denote $$d_m := \det  \big(	\operatorname{trig}(a_1,\ldots, a_m)\big).$$
	It is easy to see that $d_m$ satisfies the Fibonacci type recurrence relation $$d_m = a_m d_{m-1} - d_{m-2}, \quad \mbox{for all } m>1,$$ with $d_0 = 1$. Note that since $\rk \Lambda=n$, we have $d_{n+1}=0$.
	From this formula, we derive two things\,: first
	that $d_n \neq 0$ (otherwise $d_m$ would vanish for all $m\leq n+1$,
	but $d_0=1$)\,; second that $\gcd(d_{m},d_{m-1}) =
	\gcd(d_{m-1},d_{m-2})$ for all $2 \leq m \leq n+1$. Because $d_0 =1$
	and $d_{n+1}=0$, we see that $d_n=\pm 1$ (and in fact, since the signature is then $(n^+,n^-,0)$, $d_n = (-1)^{n^-}$). 
	\medskip

	Let us now assume that there are $\lambda_0, \lambda_1, \ldots, \lambda_{n+1}$ as in $(iii)$, with $\omega = -\sum_{i=0}^{n+1} \lambda_i$. The proof of $(iii) \Rightarrow (i)$ shows that in fact $(\lambda_1, \ldots, \lambda_{n})$ is a trigonal basis of $\Lambda$. Furthermore, $b(\omega,\lambda_i) = -b(\lambda_i,\lambda_i) - 2$ for all $1\leq i \leq n$. Therefore $\omega$ is a special characteristic element in $\Lambda$. 
	
	Conversely, pick a trigonal basis $(e_1,\ldots, e_n)$ of the unimodular lattice $(\Lambda,b)$ such that $b(\omega,e_i) = -b(e_i,e_i) - 2$ for all $1\leq i \leq n$. 
	 We note that $ \omega + e_1 + \cdots + e_n$ identifies with the dual of $-(e_1 + e_n)$ with respect to the unimodular symmetric bilinear form $b$. Therefore, for $e_0$ and $e_n$ as defined in the proof of $(i) \Rightarrow (iii)$, $\omega = -(e_0 + \cdots +e_n)$.
	 This concludes the proof of the proposition.
\end{proof}

We now state our main result\,; it gives a characterization of pairs consisting of a trigonal unimodular lattice of signature $(1,n-1)$ endowed with a special characteristic element.

\begin{thm} \label{T:mainappendix}
	Let $(\Lambda,b)$ be a  lattice of signature $(1,n-1)$ and let $\omega$ be a vector in $\Lambda$. The following statements are equivalent\,:
	\begin{enumerate}[(i)]
		\item $(\Lambda,b)$ is unimodular and trigonal, and $\omega$ is a special characteristic element\,; \vspace{4pt}
		\item The vector $\omega$ is characteristic of norm $b(\omega,\omega)=10-n$ and the pair $(\Lambda,\omega)$ satisfies one of the following properties\,:
		\begin{itemize}
			\item $\Lambda \cong \langle 1 \rangle$ and $\omega = 3\lambda$ for some
			primitive vector $\lambda$\,;
			\item   $\Lambda \cong U$ and $\omega = 2\lambda$ for some
			primitive vector $\lambda$\,;
			\item  $\Lambda \cong \langle 1 \rangle \oplus \langle -1
			\rangle^{\oplus n}$ with $n>0$ and $\omega$ is
			primitive.
		\end{itemize}\vspace{4pt}
		\item There exists a $\Z$-basis $(e_1,\ldots,e_n)$ of $\Lambda$  such that\,: 
		\begin{itemize}
			\item $\operatorname{Mat}_{(e_i)}(b) = \operatorname{diag}(1,-1,\ldots,-1)$ if $b$ is odd\,;
			\item   $\operatorname{Mat}_{(e_i)}(b) = \left(\begin{array}{cc}
			0 & 1  \\
			1 &0\end{array} \right)$ if $b$ is even\,;
		\end{itemize}
		and such that $b(\omega,e_i) = - b(e_i,e_i)-2$ for all $1\leq i \leq n$.
	\end{enumerate}
\end{thm}

Before we prove the theorem, we state and prove a few useful lemmas that give constraints on the coefficients of trigonal unimodular matrices.

\begin{lem}\label{lem -1 0}
	Let $n\geq 2$ and let $a_1,a_2,\ldots, a_n$ be real numbers. Consider a
	trigonal
	matrix $ M = [a_1,a_2,\ldots, a_n]$ as in \eqref{eq matrix}. Assume that $\det M
	= \pm 1$. Then there exists an index $i$
	such that $|a_i|<2$.
\end{lem}
\begin{proof} 
Assume for contradiction that for all $i$ we have $|a_i| \geq 2$.	Let $V$ be an $n$-dimensional $\R$-vector space with basis $(e_1,\ldots,e_n)$\,;
	we view $M$ as the matrix of a symmetric bilinear form $b$ on $V$ expressed in
	the
	basis $(e_1,\ldots, e_n)$.
	Since $|a_i| \geq 2$ for all $i$, it is possible to define inductively
	$b_1:=a_1$ and $b_i := a_i-\frac{1}{b_{i-1}}$ for $i\geq 2$\,;  in fact one sees
	by induction that $|b_i|>1$ for all $i$.  Let us consider
	the following volume-preserving change of basis for $V$\,:
	$e'_1 := e_1$, and $e'_i :=
	e_i-\frac{1}{b_{i-1}}e'_{i-1}$. 
	Expressed in the basis $(e'_1,\ldots, e'_n)$, the matrix $M'$ of the
	bilinear form
	$b$ has determinant $\det M' =\det M =\pm 1$. Moreover, $M'$ is diagonal with
	diagonal terms given by
	the real numbers $b_i$. Thus $|\det M'| = \prod_i |b_i| > 1$, which is a
	contradiction.
\end{proof}

\begin{obs}\label{obs}
Let $(\Lambda,b)$ be a trigonal unimodular lattice of rank $n$ equipped with a special characteristic element $\omega$. Consider a basis $(e_1,\ldots, e_n)$ of $\Lambda$ such that
$\operatorname{Mat}_{(e_i)}(b) = [a_1,\ldots,a_n]$ for some integers $a_i$, and such that  
$b(\omega,e_i) = -b(e_i,e_i)-2$ for all $1\leq i \leq n$.
We make the following observations\,: \medskip

$(i)$ Suppose there exists $j<n$ such that $a_j = 0$ and, given an integer $x$, consider the new basis $$(e'_i)_{1\leq i \leq n} := (e_1, \ldots, e_{j},e_{j+1}+xe_j, e_{j+2}, \ldots, e_n).$$ Then that new basis is trigonal for $b$\,; in fact the matrix of $b$ in that basis is $$\operatorname{Mat}_{(e'_i)}(b) =  [a_1,\ldots,a_{j-1},0,a_{j+1}+2x,a_{j+2},\ldots,a_n].$$ Moreover, one readily checks that $\omega$ satisfies	$b(\omega,e'_i) = -b(e'_i,e'_i)-2$ for all $1\leq i \leq n$.
\medskip

$(ii)$ Suppose there exists $j<n$ such that $a_j = -1$ and consider the new basis $$(e'_i)_{1\leq i \leq n} := (e_1, \ldots, e_{j},e_j+e_{j+1}, e_{j+2}, \ldots, e_n).$$ Then in that new basis  $b$ splits as the direct orthogonal sum of two trigonal lattices\,;  precisely the matrix of $b$ in that basis is $$\operatorname{Mat}_{(e'_i)}(b) =  [a_1,\ldots,a_{j-1},-1] \oplus [a_{j+1}+1,a_{j+2},\ldots,a_n].$$ Moreover, one readily checks that $\omega$ satisfies	$b(\omega,e'_i) = -b(e'_i,e'_i)-2$ for all $1\leq i \leq n$.
\end{obs}

\begin{lem} \label{L:0-1}  
	Suppose $(\Lambda,b)$ is a trigonal unimodular lattice and let $(e_i)_{1\leq i\leq n}$ be a  trigonal basis for $b$. 
	\begin{enumerate}[(i)]
\item If $b$ has signature $(1,n-1)$, then there
exists an index $i$
such that $a_i := b(e_i,e_i) = -1$ or $0$\,; 
\item If $b$ is positive definite, then there exists an index $i$ such that $a_i:= b(e_i,e_i) =1$\,;
\item If $b$ is negative definite, then there exists an index $i$ such that $a_i:= b(e_i,e_i) =-1$.
	\end{enumerate}
\end{lem}	
\begin{proof}
	By Lemma \ref{lem -1 0}, there exists an index $i$ such that $a_i = -1,0$ or $1$. 
	The assertions \emph{(ii)} \emph{(iii)} are then clear. 
	Suppose now that  $b$ has signature $(1,n-1)$. Assume for contradiction
	that there is no index $i$ for which $a_i=-1$ or $0$. Then by Lemma \ref{lem -1 0} there is
	an index $i$ for which $a_i =1$.
	The $(n-1)$-tuple
	$(e_1,\ldots,e_{i-2},e_{i-1} - e_i, e_i - e_{i+1}, -e_{i+2}, \ldots,-e_n)$ gives a
	$\Z$-basis of the orthogonal complement $\langle e_i \rangle^\perp$ in $\Lambda$
	of the sub-lattice spanned by $e_i$. The matrix of the bilinear form
	$b|_{\langle e_i \rangle^\perp}$ expressed in that basis is $$[a_1,\ldots,
	a_{i-2},a_{i-1}-1,a_{i+1}-1,a_{i+2},\ldots,a_n].$$ Moreover, $b|_{\langle e_i
		\rangle^\perp}$ is unimodular and negative definite. Therefore, by \emph{(iii)}, we have $-1 \in \{a_1,\ldots,
	a_{i-2},a_{i-1}-1,a_{i+1}-1,a_{i+2},\ldots,a_n\}$, and this shows that there is
	an index $j$ such that $a_j = -1$ or $0$, which contradicts our assumption.
\end{proof}

\begin{lem}\label{lem even2} Let $(\Lambda,b)$ be an even trigonal unimodular lattice of rank $n$.
	Then $n$ is even and $\Lambda \cong U^{\oplus m}$, where $2m=n$. Moreover, there exists a trigonal $\Z$-basis $(e_i)_{1\leq i\leq 2m}$ such that $\operatorname{Mat}_{(e_i)}(b) = [0,0,\ldots,0]$.  
\end{lem}
\begin{proof} Let  $(e_i)_{1\leq i\leq n}$ be a trigonal basis for $b$ and denote $a_i := b(e_i,e_i)$. By Lemma \ref{lem -1 0}, for the bilinear form $b$ to be
	unimodular, one of the $a_j$ has to be equal to $-1, 0$ or $1$. The pairing is
	assumed to be even, so that one of the $a_j$ is equal to $0$. The integer
	$a_{j+1}$ is even. Consider, as in Observation \ref{obs}$(i)$, the  change of $\Z$-basis $e_i \mapsto
	e_i$ for $i\neq j+1$ and $e_{j+1} \mapsto e_{j+1}-\frac{a_{j+1}}{2}e_j$. In that
	new basis, the matrix of the form $b$ is
	$[a_1,\ldots,a_{j-1},0,0,a_{j+2},\ldots, a_{2n}]$. Performing several similar
	changes of bases shows that there is a basis $(e'_1,\ldots, e'_{2n})$  of
	$\Lambda$ such that the matrix of the form $b$ expressed in that basis is
	$[0,0,\ldots,0]$.  A straightforward calculation shows that $\det b = 0$ if $n$ is odd, and that $\det b = (-1)^m$ if $n=2m$ for some integer $m$. Thus $n$ is even.
	Consider then the basis $(e'_1,e'_2,e'_3-e_1',e'_4,e_5'-e_3',e'_6, \ldots,
	e'_{2m-1}-e'_{2m-3}, e'_{2m})$ of $\Lambda$. It becomes apparent that $\Lambda
	\cong U^{\oplus m}$.
\end{proof}

The following proposition and its corollary prove $(i)\Rightarrow (ii)$ of Theorem \ref{T:mainappendix}, and is the heart of the proof of $(ii) \Rightarrow (iii)$ of our Main Theorem \ref{thm criterion} (and hence of the arithmetic application thereof given by Theorem \ref{thm geomrat surfaces}).

\begin{prop} \label{P:i-ii}
Let $(\Lambda,b)$ be a trigonal unimodular lattice of signature $(1,n-1)$ and let $\omega$ be a special characteristic element  in $\Lambda$. Then there exists a $\Z$-basis $(e_1,\ldots, e_n)$ of $\Lambda$ such that
 \begin{center}
$b(\omega,e_i) = -b(e_i,e_i)-2$ for all $1\leq i \leq n$    \quad and \quad  
$\Lambda = \left\{
\begin{array}{ll}
 \operatorname{trig}(0,0) & \mbox{if } \Lambda \mbox{ is even}\,; \\
\langle 1 \rangle \oplus \langle -1 \rangle^{\oplus n-1} & \mbox{if } \Lambda \mbox{ is odd}.
\end{array} \right.$
 \end{center}
\end{prop}
\begin{proof} The case $n=1$ is trivial. Assume that $n=2$\,; in that case any trigonal basis $(e_1,e_2)$ for $b$ is such that $\operatorname{Mat}_{(e_i)}(b) = [a_1,a_2]$ with $a_1a_2= 0$ (since $\det b = -1$). Thus, by Observation \ref{obs}$(i)$, there is a basis $(e_1,e_2)$ for $b$ is such that $\operatorname{Mat}_{(e_i)}(b) = [0,a]$ with $a = 0$ or $-1$, and such that $\omega$ is special with respect to that basis. The former case is the case where $\Lambda$ is even, while in the latter case Observation \ref{obs}$(ii)$ gives us a new basis, namely $(e_1'+e_2',e_2')$, in which the matrix of $b$ is $\operatorname{diag}(1,-1)$ and  with respect to which $\omega$ is special.
	
	For the sake of the induction argument to come, let us consider a negative definite trigonal unimodular lattice $(\Lambda,b)$  of rank $2$. Let $(e_1,e_2)$ be a trigonal basis\,; the unimodularity of $b$ shows that up to reordering $e_1$ and $e_2$, we have $\operatorname{Mat}_{(e_i)}(b) = [-2,-1]$. Assume that $\omega$ is such that $b(\omega,e_i) = -b(e_i,e_i)-2$. Then Observation \ref{obs}\emph{(ii)} says that in the basis $(e'_1,e'_2) := (e_1+e_2,e_2)$ we have $\operatorname{Mat}_{(e'_i)}(b) = \operatorname{diag}(-1,-1)$ and $b(w,e'_i) = -b(e_i,e_i) - 2 = -1$. 
	
Assume now that $n \geq 3$\,; we are going to proceed by induction. 
We suppose that for all $m<n$, if $(\Lambda,b)$ is a trigonal unimodular lattice of signature $(1,m-1)$ or $(0,m)$ endowed with a special characteristic element $\omega$, then there exists a basis $(e_1,\ldots, e_m)$ of $\Lambda$ such that $b(\omega,e_i) = -b(e_i,e_i)-2$ for all $1\leq i \leq m$, and such that $\operatorname{Mat}_{(e_i)}(b)$ is either equal to $[0,0]$ or to $\operatorname{diag}(\pm 1, -1,\ldots, -1)$.
We now fix a trigonal unimodular lattice $(\Lambda,b)$ of signature $(1,n-1)$ or $(0,n)$ endowed with a special characteristic element $\omega$.
By Lemma \ref{lem even2}, the trigonal lattice $(\Lambda,b)$, which has signature $(1,n-1)$ or $(0,n)$, has to be odd. Let
$(e_1,\ldots, e_n)$ be a trigonal basis of $\Lambda$ with respect to which $\omega$ is special. 
 By Lemma \ref{L:0-1}(i), there exists $1\leq j \leq n$, such that $b(e_j,e_j) = -1$ or $0$. Suppose that $b(e_j,e_j) = 0$. Then, by repeated use of Observation \ref{obs}$(i)$, we find a trigonal basis $(e'_1,\ldots,e'_n)$ of $\Lambda$ with respect to which $\omega$ is special, and such that $b(e'_k,e'_k)=-1$ for some $k$. Therefore,  we may assume that $b(e_j,e_j) = -1$ in the first place. By Observation \ref{obs}$(ii)$, we find a basis $(f_1,\ldots, f_n)$ of $\Lambda$ such that  \begin{center}
 	$b(\omega,f_i) = -b(f_i,f_i)-2$  \quad and \quad  
 	$\Lambda = [b(f_1,f_1),\ldots,b(f_{j-1},f_{j-1})] \oplus \langle -1 \rangle \oplus [b(f_{j+1},f_{j+1}),\ldots,b(f_n,f_n)].$
\end{center}
By the induction hypothesis, we obtain a basis $(f'_1,\ldots, f'_n)$ for which $b(\omega,f'_i) = -b(f'_i,f'_i)-2$ for all $1 \leq i \leq m$, and for which $\Lambda$ is either $\langle -1 \rangle^{\oplus n}$, $\langle 1 \rangle \oplus \langle -1 \rangle^{\oplus n-1}$ or $ \langle -1 \rangle^{\oplus n-2} \oplus U$. In order to finish off the induction, we note that if there is a basis $(e_1,e_2,e_3)$ of a rank-$3$ unimodular lattice $\Lambda$ such that $\operatorname{Mat}_{(e_i)}(b) = \langle -1 \rangle \oplus [0,0]$ with $b(\omega,e_i) = -b(e_i,e_i)-2$ for $1\leq i \leq 3$, then in the basis $$(e'_i)_{1\leq i \leq 3} := (e_2+e_3-e_1, e_2-e_1, e_3-e_1)$$ we have $\operatorname{Mat}_{(e'_i)}(b) = \langle 1 \rangle \oplus \langle -1 \rangle \oplus \langle -1 \rangle$ with $b(\omega,e'_i) = -b(e'_i,e'_i)-2$ for $1\leq i \leq 3$.
\end{proof}

\begin{cor} \label{C:pq}
	Let $(\Lambda,b)$ be a trigonal unimodular lattice of signature $(1,n-1)$. If $\omega$ is a special characteristic element  in $\Lambda$, then 
	\begin{equation}\label{eq kappa2}
		b(\omega,\omega) = 10-n.
	\end{equation}
\end{cor}
\begin{proof}
	This follows immediately from Proposition \ref{P:i-ii}\,: 
		\begin{itemize}
			\item If $\Lambda = [0,0]$, then $\omega = -2e_1 - 2e_2$ and hence $b(\omega,\omega) = 8$.
			\item If $\Lambda = 
			\langle 1 \rangle \oplus \langle -1 \rangle^{\oplus n-1}$, then $\omega = -3e_1 - e_2 - \cdots - e_n$ and hence $b(\omega,\omega) = 10-n$.
		\end{itemize}
\end{proof}

\begin{rmk} \label{rmk K2} 
	   In fact, it is possible to generalize Proposition \ref{P:i-ii} to unimodular lattices of any signature\,: it can be shown that if $\Lambda$ is unimodular and trigonal, then there exists a basis $(e_1,\ldots, e_n)$ of $\Lambda$ such that $b(\omega,e_i) = -b(e_i,e_i)-2$ for all $1\leq i \leq n$, and such that with respect to that basis $\Lambda$ splits as a direct orthogonal sum of lattices isomorphic to $[0,0]$, $[1]$,$[-1]$, $[1,2]$ and $[1,3,1]$. (Note that in particular a unimodular lattice is trigonal if and only if it is isomorphic to a direct sum of copies of $U$, $\langle 1 \rangle$, and $\langle -1 \rangle$.)
	 As a consequence, if $(\Lambda,b)$ is a trigonal unimodular lattice of signature $(n^+,n^-)$ and if $\omega$ is a special characteristic element  in $\Lambda$, then 
		\begin{equation}\label{eq kappa}
		b(\omega,\omega) = 8\left\lfloor \frac{n^++1}{2}\right\rfloor
		+n^+-n^-.
		\end{equation}
	The formula \eqref{eq kappa} should be compared to van der Blij's lemma
	\cite[Lemma	II.(5.2)]{hm}. 
Let $(\Lambda, b)$ be a unimodular lattice\,; it is easy to see
	that a characteristic element $\omega$ always exists since the function $\Lambda \rightarrow
	\Z/2\Z, \lambda \mapsto b(\lambda, \lambda) [\mbox{mod}\ 2]$ is $\Z/2\Z$-linear. It is
	also easy to check that the integer $b(\omega,\omega)$ is an invariant
	modulo $8$. Van der Blij's lemma states that in fact $b(\omega,\omega) =
	n^+-n^-\ [\mbox{mod}\ 8].$
	Thus, for a trigonal unimodular lattice, a special characteristic element $\omega$ (that is, the element $\omega \in \Lambda$ such that $b(\omega, e_i) = -b(e_i,e_i)-2$ for a basis $(e_1,\ldots,e_n)$ of $\Lambda$ in which the matrix of $b$ is trigonal)  can be thought of as an integral characteristic element in the lattice $\Lambda$, and
	\eqref{eq kappa} gives an integral version of van der Blij's lemma for
	trigonal unimodular lattices.
\end{rmk}

The following Witt-type proposition proves $(ii)\Rightarrow (iii)$ of Theorem \ref{T:mainappendix} for odd unimodular lattices, and is the heart of the proof of $(iii) \Rightarrow (i)$ of our Main Theorem \ref{thm criterion} (and hence of Theorem \ref{thm surface general type}).

\begin{prop} \label{P:ii-iii}
	Let $(\Lambda,b)$ be an odd unimodular lattice of signature $(1,n-1)$ and let $\omega$ be a characteristic element in $\Lambda$ such that $b(\omega,\omega) = 10-n$. Assume further that $\omega$ is primitive if $n\geq 10$. Then there exists a basis $(e_1,\ldots, e_n)$ of $\Lambda$ such that 
	$$\operatorname{Mat}_{(e_i)}(b) = \operatorname{diag}(1,-1,\ldots,-1) \quad \mbox{and} \quad \omega = 3e_1 + e_2 + \ldots + e_n.$$
\end{prop}
\begin{proof} In the case $n>10$, that is, in the case $b(\omega,\omega)<0$, the proposition was already proved in greater generality by Nikulin \cite{nikulin}. 
	Indeed, assume $n>10$ and let $\omega$ be a primitive characteristic element of norm $b(\omega,\omega) = 10-n$ in the lattice $\Lambda$ of signature $(1,n-1)$. In particular, $b(\omega,\omega)<0$ and the restriction of $b$ to the orthogonal complement $\omega^\perp$ of $\omega$ is even and indefinite. Therefore, we may invoke \cite[Prop. 3.5.1]{nikulin} (which applies since our lattice has rank $\geq 4$), which says that there is only one orbit under $O(q)$ of primitive characteristic elements of given negative norm.\medskip
	
Consider now  an odd unimodular lattice $(\Lambda,b)$ of signature $(1,n-1)$. Pick a basis $(e_1,\ldots, e_n)$ of $\Lambda$ such that 	$\operatorname{Mat}_{(e_i)}(b) = \operatorname{diag}(1,-1,\ldots,-1)$. \medskip

\noindent \textbf{Claim 1.} Let  $\omega$ be an element of $\Lambda$ such that $b(\omega,\omega) \geq 0$. Then there is an automorphism $\varphi$ of $\Lambda$ preserving $q$ (\emph{i.e.}, $\varphi \in O(q)$) such that $\varphi(\omega) = x_1e_1 + \cdots + x_ne_n$ with 
\begin{equation}\label{E:ineq}
0\leq x_n \leq x_{n-1} \leq \cdots \leq x_1 \quad \mbox{and} \quad x_4 + x_3 + x_2 \leq x_1.
\end{equation}
(When $n=1$ or $2$, the latter inequality should be ignored\,; and when $n=3$, it should be understood to read $x_3 + x_2 \leq x_1$).

\noindent In other words, the orbit of any element of non-negative norm under the action of $O(q)$ contains an element whose coordinates satisfy \eqref{E:ineq}. 

\begin{proof}[Proof of Claim 1.]
Given a vector $v \in \Lambda$ of norm $b(v,v)$ that divides $2$, we define the reflection $R_v \in O(q)$ across the hyperplane orthogonal to $v$ by the formula $$R_v(\lambda) := \lambda - 2\frac{b(\lambda,v)}{b(v,v)} v.$$

Let $\xi := x_1e_1 + \cdots + x_ne_n$ be an element of $\Lambda$.  Up to applying the reflections $R_{e_i}$, we see that all vectors $\pm x_1e_1 \pm \cdots \pm x_ne_n$ belong to the orbit of $\xi$ under the action of $O(q)$. Consider the action of the symmetric groups $\mathfrak{S}_{n-1}$ on the set $\{2,\ldots, n\}$\,; this induces an action on the lattice $\Lambda$, given by $f_{\sigma}(\xi) := x_1e_1 + x_{\sigma^{-1}(2)}e_2 +\cdots + x_{\sigma^{-1}(n)}e_n$ for all $\sigma \in \mathfrak{S}_{n-1}$. Clearly $\sigma \mapsto f_\sigma$ defines a homomorphism $\mathfrak{S}_{n-1} \to O(q)$. Therefore, $f_{\sigma}(\xi)$ belongs to the orbit of $\xi$ for all $\sigma \in \mathfrak{S}_{n-1}$.

Consider now  a non-zero element $\xi$  of non-negative norm. By the above, the vector $\xi$ has in its orbit a vector $x_1e_1 + \cdots + x_ne_n$ with $0\leq x_1$ and $0\leq x_n\leq x_{n-1} \leq \cdots \leq x_2$. Since $b(\xi,\xi) \geq 0$, we actually have $ 0\leq x_n  \leq \cdots \leq x_1$. 
Choose now such a vector in the orbit of $\xi$ with minimal non-negative $x_1$. Assume that $n\geq 4$ (we indicate how to treat the case $n\leq 3$ at the end of the proof). We claim that $x_4 + x_3 + x_2 \leq x_1$. If that is not case, consider the reflection $R_v$ with $v := e_1+e_2+e_3+e_4$. Then we have 
\begin{align*}
R_v(\xi) =& (2x_1-x_2-x_3-x_4)e_1 + (x_1-x_3-x_4)e_2 + (x_1-x_2-x_4)e_3 + (x_1-x_2-x_3)e_4\\
& + e_5 + \cdots + e_n.
\end{align*}  Given that $ 0\leq x_4 \leq x_{3} \leq x_2 \leq x_1$ and $x_4 + x_3 + x_2 > x_1$ (note also that $x_4 < x_1$ because $b(\xi,\xi) \geq 0$), we have $$-x_1 < 2x_1 - x_2 - x_3 - x_4 < x_1.$$ Therefore,  making all the coordinates of $R_v(\xi)$ non-negative and reordering them in decreasing order, we obtain a vector $\xi' := x'_1e_1 + \cdots + x'_ne_n$ in the orbit of $\xi$ with $0 \leq x'_n \leq \cdots \leq x'_1 < x_1$, thus yielding a contradiction.

\noindent (The cases $n=1$ and $n=2$ are obvious, while in the case $n=3$ one proceeds similarly by considering the reflection $R_v$ with $v := e_1+e_2+e_3$.)
\end{proof}

\noindent \textbf{Claim 2.} Assume $n<10$. Then there is only one orbit of characteristic elements of $\Lambda$ of norm $10-n$\,; it is the orbit of the element $3e_1+e_2 +\cdots +e_n$.  

\begin{proof}[Proof of Claim 2.]
For this purpose, given Claim 1, we show that a characteristic vector $\xi := x_1e_1 + \cdots + x_ne_n$ of norm $b(\xi,\xi)=10-n$ and whose coordinates satisfy \eqref{E:ineq}  is necessarily the vector $3e_1+e_2+ \cdots +e_n$. 
Squaring the inequality $x_4 + x_3 + x_2 \leq x_1$ yields 
\begin{equation}\label{E:ineq2}
2(x_2x_3 + x_2x_4 + x_3x_4) \leq x_1^2-x_2^2-x_3^2-x_4^2 = 10-n +x_5^2+\cdots + x_n^2.
\end{equation} Using the comparison of $x_i$ with $x_4$, we find $6x_4^2 \leq 10-n + (n-4)x_4^2$. The assumption $n<10$ immediately gives $x_4 \leq 1$. Since we are assuming that $\xi$ is characteristic, this forces $x_n = \cdots = x_4 =1$.  Therefore we obtain $x_1^2-x_2^2-x_3^2 = 7$. Squaring the inequality $x_2 + x_3 < x_1$ then gives $2x_2x_3 <7$. Thus $(x_2,x_3)$ is either $(1,1)$ or $(3,1)$ (recall that $\xi$ is characteristic so that its coordinates are odd integers). But the latter is not possible since otherwise $17$ would be a square. Hence $(x_1,x_2,x_3) = (3,1,1)$.  
\end{proof}

\noindent \textbf{Claim 3.} Assume $n=10$. Then there is only one orbit of isotropic primitive characteristic elements of $\Lambda$\,; it is the orbit of the element $3e_1+e_2 +\cdots +e_n$.  
\begin{proof}[Proof of Claim 3.]
As in the proof of Claim 2, we obtain the inequality \eqref{E:ineq2}. Singling out the term $x_{10}^2$, we obtain $$6x_4^2 \leq  5x_4^2 + x_{10}^2$$ and hence $x_4^2\leq x_{10}^2$. This proves that $x_4 = x_5 =\cdots = x_{10}$. We are thus reduced to solving the Diophantine equation $$x_1^2 = x_2^2 + x_3^2 + 7x_4^2$$ with the constraint that $x_1,\ldots, x_4$ are odd integers, with no common prime factors, satisfying $0<x_4 \leq x_3 \leq x_2 \leq x_1$ and $x_4+x_3 + x_2 \leq x_1$. On the one hand, squaring the latter inequality  gives $$2(x_2x_3 + x_2x_4 + x_3x_4) \leq x_1^2-x_2^2-x_3^2-x_4^2 = 6x_4^2.$$ On the other hand, the former inequality gives $$6x_4^2 \leq 2(x_2x_3 + x_2x_4 + x_3x_4) $$ with equality if and only if $x_2 = x_3 = x_4$. Therefore, we immediately get $x_2 = x_3 = x_4$ and then that $x_1 = 3x_2$. The only primitive solution is then $(3,1,1,1)$.
\end{proof}
The proof of Proposition \ref{P:ii-iii} is now complete.
\end{proof}

Finally we provide a proof of Theorem \ref{T:mainappendix}.

\begin{proof}[Proof of Theorem \ref{T:mainappendix}]
	$(i)\Rightarrow (ii)$\,: Given Proposition \ref{P:i-ii} and Corollary \ref{C:pq}, it only remains to see that $\omega = -3e_1$ if $\Lambda = \langle 1 \rangle$, $\omega = -2e_1-2e_2$ if $\Lambda = [0,0]$, and that $\omega = -3e_1 - e_2 - \cdots - e_n$ if $\Lambda = \langle 1 \rangle \oplus \langle -1 \rangle^{\oplus n-1}.$
	
	$(ii) \Rightarrow (iii)$\,: Given Proposition \ref{P:ii-iii}, it only remains to treat the case where $\Lambda$ is isomorphic to the hyperbolic plane and $\omega$ is twice a primitive vector. Let $(e_1,e_2)$ be a basis of $\Lambda$ such that the matrix of $b$ is $[0,0]$. Let us then write $\omega = a e_1 +b e_2$. By assumption we have $b(\omega,\omega) = 2ab = 8$. Since  $\omega$ is twice a primitive vector, we find up to considering the new basis $(\pm e_1,\pm e_2)$ that $\omega = -2e_1 - 2e_2$. It is then apparent that  $b(\omega,e_i) = - b(e_i,e_i)-2$ for $1\leq i \leq 2$.
	
		$(iii) \Rightarrow (i)$\,: The even case is obvious, and so is the case where $\Lambda = \langle 1 \rangle$. 	Let us thus consider an odd unimodular lattice $(\Lambda,b)$ of signature $(1,n-1)$ with $n>1$, and let  $(e_1,\ldots, e_n)$ be a basis of $\Lambda$ such that 
		$\operatorname{Mat}_{(e_i)}(b) = \operatorname{diag}(1,-1,\ldots,-1) $. Let $\omega$ be the element in $\Lambda$ such that $b(\omega,e_i) = -b(e_i,e_i)-2$ for all $1\leq i \leq n$. Then one readily checks that the vectors 	$$	\left\{
		\begin{array}{lll} e'_i &=& e_1 - e_{i+1}, \quad 1\leq i <  n\\
		e'_n &=& e_n \end{array} \right.$$ provide a basis $(e'_1,\ldots, e'_n)$  of $\Lambda$ such that 
		$\operatorname{Mat}_{(e'_i)}(b) = [-1,0,\ldots,0]$ and such that $b(\omega,e'_i) = -b(e'_i,e'_i)-2$ for all $1\leq i \leq n$. 
\end{proof}



\end{document}